\documentclass{amsart} 
\usepackage{amsmath,amssymb,amsthm}
\usepackage{hyperref}
\usepackage{xcolor}
\usepackage{enumerate}
\usepackage{graphicx}
\usepackage{fancyhdr}
\usepackage{amsfonts,latexsym}
\usepackage{stmaryrd}

\calclayout

\def\p{\partial}

\newcommand{\R}{{\mathbb R}}

\newcommand{\C}{{\mathbb C}}

\newcommand{\supp}{\operatorname{supp}}
\newcommand{\diam}{\operatorname{diam}}
\newcommand{\dist}{\operatorname{dist}}

\newcommand{\mathbi}[1]{{\boldsymbol{#1}}}

\newcommand{\mm}{{-}}
\newcommand{\pp}{{+}}

\newcommand{\frc}{{\mathrm{f}}}

\newcommand{\jmp}[1]{{\,\big[\,#1\,\big]_\mm^\pp }}

\newcommand{\TT}{\parallel}
\newcommand{\nablaS}{\nabla^{^\Surf}\!}

\newcommand{\Surf}{{\Sigma}}

\newcommand{\fSurf}{\Surf_\mathrm{f}}

\newcommand{\itLa}{\mathit{\Lambda}}
\newcommand{\itLaINI}{\itLa^{\tsT^0}}

\newcommand{\tsC}{\mathbi{C}}
\newcommand{\tsLa}{\mathbi{\mathit{\Lambda}}}
\newcommand{\tsLaINI}{\tsLa^{\!\tsT^0\!}}

\newcommand{\tsT}{\mathbi{T}}

\newcommand{\vef}{\mathbi{f}}

\newcommand{\veu}{\mathbi{u}}
\newcommand{\ven}{\mathbi{n}}

\newcommand{\vetau}{{\mathbi{\tau}}}

\def\cbeta{c_4}

\theoremstyle{definition}

\newtheorem*{construction*}{Construction}
\newtheorem*{notation*}{Notations}
\newtheorem{remark}{Remark}

\theoremstyle{theorem}

\newtheorem{theorem}{Theorem}[section]
\newtheorem{lemma}[theorem]{Lemma}

\newtheorem{proposition}[theorem]{Proposition} 
\newtheorem{corollary}[theorem]{Corollary} 
\newtheorem{assumption}[theorem]{Assumption}

\usepackage[bordercolor=white, color=white, textwidth=100pt]{todonotes}

\numberwithin{equation}{section}

\begin{document} 

\title[Quantitative unique continuation and the kinematic inverse rupture problem]{Quantitative unique continuation for the elasticity system with application to the kinematic inverse rupture problem}

\author[M. V. de Hoop]{Maarten V. de Hoop}
\address{Maarten V. de Hoop: Computational and Applied Mathematics and Earth Science, Rice University, Houston, TX 77005, USA}
\email{mdehoop@rice.edu}

\author[M. Lassas]{Matti Lassas}
\address{Matti Lassas: Department of Mathematics and Statistics, University of Helsinki, FI-00014 Helsinki, Finland}
\email{matti.lassas@helsinki.fi}

\author[J. Lu]{Jinpeng Lu}
\address{Jinpeng Lu: Department of Mathematics and Statistics, University of Helsinki, FI-00014 Helsinki, Finland} 
\email{jinpeng.lu@helsinki.fi}

\author[L. Oksanen]{Lauri Oksanen}
\address{Lauri Oksanen: Department of Mathematics and Statistics, University of Helsinki, FI-00014 Helsinki, Finland} 
\email{lauri.oksanen@helsinki.fi}



\maketitle 

\vspace{-5mm}
\begin{abstract}
We obtain explicit estimates on the stability of the unique continuation for a linear system of hyperbolic equations. In particular our result applies to the elasticity system and also the Maxwell system. 
As an application, we study the kinematic inverse rupture problem of determining the jump in displacement and the friction force at the rupture surface, and we obtain new features on the stable unique continuation up to the rupture surface.
\end{abstract}

\section{Introduction}


The unique continuation property for a differential operator $P$ states the following: given an open set $\Omega\subset \R^{n+1}$ and a small subset $U\subset \Omega$, if $Pu=0$ and $u|_U=0$, then $u=0$ in $\Omega$. 
Holmgren's Theorem states that for operators with analytic coefficients, the local version of the unique continuation property holds across any non-characteristic hypersurface.
For operators with only smooth coefficients, the local unique continuation across a hypersurface holds if the hypersurface satisfies a pseudoconvexity condition \cite{H63}. 
Such pseudoconvexity condition cannot be dropped due to the existence of counterexamples given by \cite{Alinhac}.
For operators with coefficients that are analytic in part of the variables, for instance the wave operator with coefficients analytic in time, Tataru proved in the seminal paper \cite{Tataru1} that the local unique continuation property holds across any non-characteristic hypersurface, which leads to a global unique continuation result in optimal time.
Tataru's unique continuation theorem is crucial for the Boundary Control method in solving inverse problems for linear equations, see e.g. \cite{AKKLT,BK,BKL3,BILL,KKL,KrKL,KOP}.
The unique continuation for linear systems of hyperbolic equations was studied in \cite{EINT}, and the result can be applied to the time-dependent classical elasticity system and the Maxwell system.

We are interested in the stability of the unique continuation: if $Pu$ is small in $\Omega$ and $u$ is small in $U$, then $u$ is small in $\Omega$.
Inspired by Tataru's ideas in \cite{Tataru-unpublished}, the quantitative stability of the unique continuation for the wave operator was obtained by \cite{BKL2,BKL1} and \cite{LL} independently.
An explicit stability of the unique continuation on Riemannian manifolds with boundary was recently obtained by \cite{BILL}, with $U$ being a subset of the boundary. In this paper, we study the explicit stability of the unique continuation for a linear system of hyperbolic equations on Riemannian manifolds with boundary, with $U$ being an interior open subset of the manifold. 
In particular, our result provides an explicit stability of the unique continuation for the classical elasticity system and the Maxwell system.

\subsection{Main results}

We consider a linear system of hyperbolic equations on $\R\times \R^n$ of the type
\begin{equation}\label{eq_system}
P_i u_i+L_i (D u,u)=f_i,\quad  i=1,\cdots,m,
\end{equation}
where $P_i$ is the wave operator with time-independent wave speed $v_i$:
\begin{equation}\label{def-wave}
P_i= \partial_t^2 - v_i(x)^2 \Delta_g, \quad \Delta_g=\sum_{j,k=1}^n g^{jk}(x) D_j D_k+\sum_{j=1}^n h_j(x)D_j+q(x),
\end{equation}
and $L_i$ are linear functions of $D u_k,u_k$ ($k=1,\cdots,m$)
with time-independent $L^{\infty}(\R^n)$ coefficients. We denote by $\|L\|_{\infty}$ the maximum over $i$ of the $L^{\infty}(\R^n)$-norms of the coefficients of $L_i$.
Assume that $v_i\in C^1(\R^n)$, $v_i>0$, $g^{jk}\in C^1(\R^n)$ and $h_j,q\in C^{0}(\R^n)$. We write 
$$u=(u_1,\cdots,u_m), \quad f=(f_1,\cdots,f_m),$$
and denote
$$\|u\|_{L^2(\Omega)}^2=\sum_{i=1}^m \|u_i\|_{L^2(\Omega)}^2, \quad \|u\|_{H^1(\Omega)}^2=\sum_{i=1}^m \|u_i\|_{H^1(\Omega)}^2.$$
In particular, it was shown in \cite{EINT} that the elasticity system and also the Maxwell system can be written in the form of hyperbolic equations \eqref{eq_system}.

\smallskip
On a Riemannian manifold $(M,g)$, we consider wave operators \eqref{def-wave} with coefficients $g^{jk}$ locally given by the Riemannian metric $g$. 
More precisely, the matrix $(g^{jk})$ is the inverse of the Riemannian metric $(g_{jk})$ in local coordinates.
In particular, one can consider $\Delta_g$ to be the Laplace-Beltrami operator on $(M,g)$.
Our main result is the following explicit stability estimate for the unique continuation for the system \eqref{eq_system}.

\begin{theorem}\label{uc-manifold}
Let $(M^n,g)$ be a compact, orientable, smooth Riemannian manifold of dimension $n\geq 2$ with smooth boundary $\partial M$, and $U$ be a connected open subset of $M$ with smooth boundary $\partial U$. Assume $\overline{U}\cap \partial M=\emptyset$.
Suppose $u=(u_1,\cdots,u_m)$, $u_i \in H^1(M\times [-T,T])$ is a solution of the system of hyperbolic equations \eqref{eq_system} with $f=(f_1,\cdots,f_m),\,f_i \in L^2(M\times [-T,T])$. Let $v=\min\limits_i \inf\limits_{x\in M} v_i(x)>0$ be the minimal wave speed in $M$.
If
$$\|u\|_{H^1(M\times [-T,T])}\leq \Lambda_0,\quad \|u\|_{H^1(U\times [-T,T])}\leq \varepsilon_0,$$
then there exist constants $h_0,C,c>0$ such that for any $0<h< h_0$, we have
$$\|u\|_{L^2(\Omega_v(h))}\leq C \exp(h^{-c n}) \frac{\Lambda_0}{\Big(\log\big(1+\frac{h^{-1}\Lambda_0}{\|f\|_{L^2(M\times [-T,T])}+h^{-2}\varepsilon_0}\big)\Big)^{\frac12}}\, .$$
The domain $\Omega_v (h)$ is defined by
\begin{equation}\label{Omega-UTv}
\Omega_{v}(h)=\big\{(x,t)\in (M\setminus U)\times [-T,T]: vT-v|t|-d(x,\partial U) >\sqrt{h},\; d(x,\partial M)>h \big\},
\end{equation}
where $d$ denotes the Riemannian distance of $M$. The constants $h_0,C$ depend on $n,m,T,v,$ $\max_i\|v_i\|_{C^1}, \|L\|_{\infty}$ and geometric parameters; $c$ is an absolute constant.
\end{theorem}

Theorem \ref{uc-manifold} will be proved in Section \ref{sec-manifold}, using the technical tools developed in Section \ref{section-uc}, in particular Proposition \ref{global}. An illustration and a brief discussion of the domain $\Omega_v(h)$ can be found in Figure \ref{fig_domains} below and Remark \ref{remark-domains}.

\begin{figure}[h]
  \begin{center}
    \includegraphics[width=0.9\linewidth]{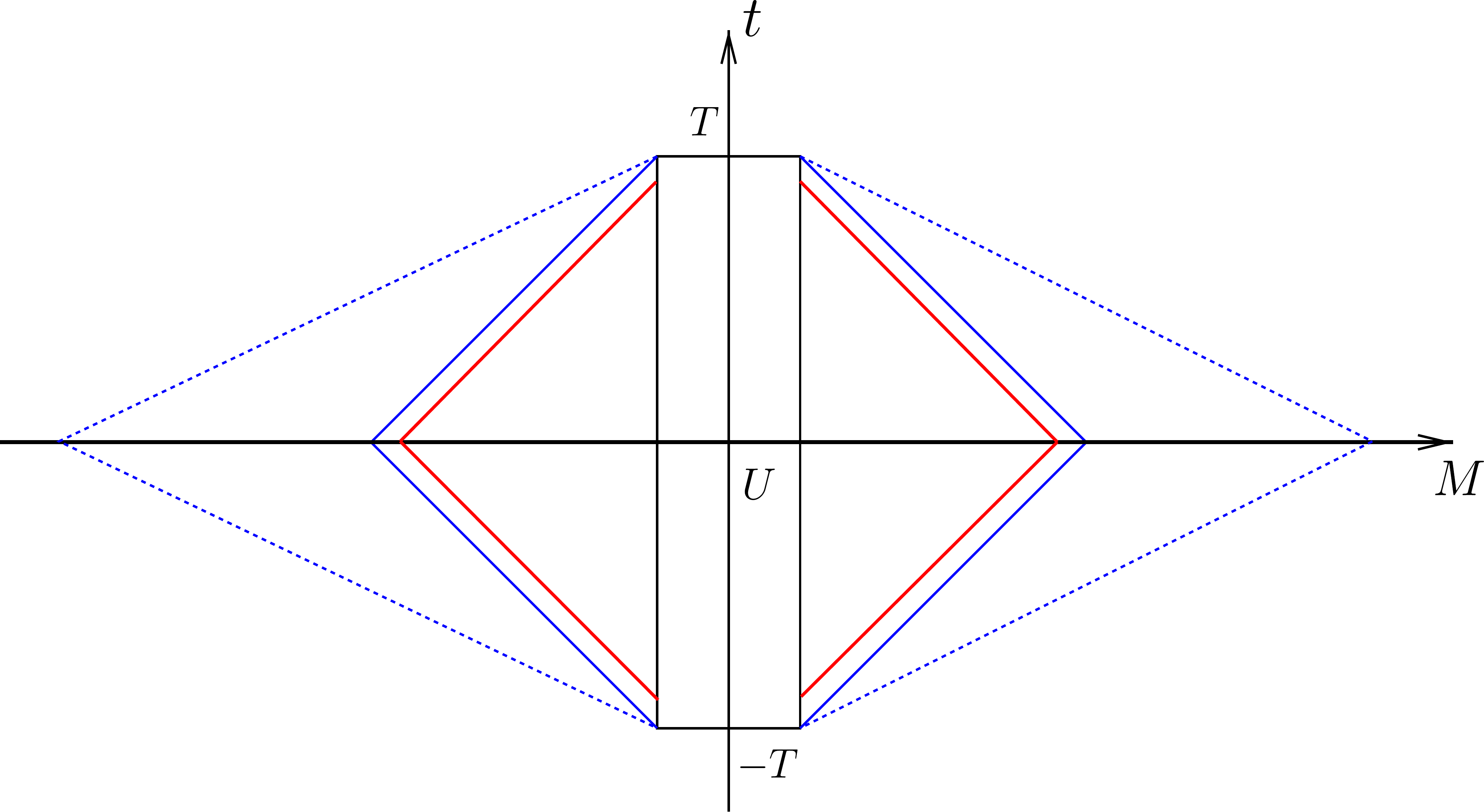}
    \caption{An illustration of domains in $1+1$ dimension.
    The rectangle in the middle is $U\times [-T,T]$.
    In the case of wave speeds $v_i\equiv 1$, the domain enclosed by the solid blue lines is the optimal domain, and the domain enclosed by the red lines is $\Omega_v(h)$ defined in \eqref{Omega-UTv}.
    The distance between the blue and red lines is of order $\sqrt{h}$.
     In general, if the wave speeds are not constant, the domain $\Omega_v(h)$ can be significantly smaller than the optimal domain. 
     For instance, in the case of two equations with $v_1\equiv 1$ and $v_2\equiv 2$, the dashed blue lines enclose the optimal domain for the scalar wave equation with wave speed $2$, while the domain $\Omega_v(h)$, propagating according to the slower speed, is still enclosed by the red lines.
     }
    \label{fig_domains}
  \end{center}
\end{figure}

\smallskip
Theorem \ref{uc-manifold} yields the following stable continuation result on the whole manifold.

\begin{corollary} \label{uc-whole-loglog}
Let $(M^n,g)$ be a compact, orientable, smooth Riemannian manifold of dimension $n\geq 2$ with smooth boundary $\partial M$, and $U$ be a connected open subset of $M$ with smooth boundary $\partial U$. Assume $\overline{U}\cap \partial M=\emptyset$.
Suppose $u=(u_1,\cdots,u_m)$, $u_i \in H^1(M\times [-T,T])$ is a solution of the system of hyperbolic equations \eqref{eq_system} with $f=(f_1,\cdots,f_m)=0$. 
Assume $T>2({\rm diam}(M)+1)/v$, where $v=\min\limits_i \inf\limits_{x\in M} v_i(x)>0$ is the minimal wave speed in $M$. 
If
$$\|u\|_{H^1(M\times [-T,T])}\leq \Lambda_0,\quad \|u\|_{H^1(U\times [-T,T])}\leq \varepsilon_0,$$
then there exist constants $\widehat{\varepsilon_0},C,c>0$ such that for any $0<\varepsilon_0<\widehat{\varepsilon_0}$, we have
$$\|u \|_{L^2 ((M \setminus U)\times [-\frac{T}{2},\frac{T}{2}] )}\leq C \big(\log |\log \varepsilon_0| \big)^{-c} ,$$
where $C$ is independent of $\varepsilon_0$, and $c$ depends only on $n$.
Furthermore, for any $\theta\in (0,1)$, by interpolation,
$$\|u \|_{H^{1-\theta}((M\setminus U)\times [-\frac{T}{2},\frac{T}{2}])}\leq C \big(\log |\log \varepsilon_0| \big)^{-\theta c}.$$
\end{corollary}

\smallskip
\subsection{Kinematic inverse rupture problem}

Next, we apply our results to an elasticity system to study the \emph{kinematic inverse rupture problem} of determining the jump of particle velocity across the rupture surface and the friction force, see \cite{FSG}.
Inverse problems for elasticity systems have been extensively studied in various settings, e.g. inverse source problems \cite{BCL,BLZ,LHL}, inverse obstacle scattering \cite{DLW,LLS,LY}, seismic inverse scattering \cite{SD}, and see e.g. \cite{BFPRV,BDF,CHN,CNO,ER,MR,NT,NU1,NU2,YZ} for inverse boundary value problems of determining the elastic body, \cite{DUW,UZ} for inverse problems for nonlinear elastic wave equations, and \cite{BFKL,Kress,NUW,NW} for identifying inclusions or cracks.

\smallskip
In our setting, let $M^3\subset \mathbb{R}^3$ be a compact domain of dimension $3$ with smooth boundary representing the solid Earth. Let $\Surf_\frc$ be a ($2$-dimensional) smooth rupture surface satisfying $\overline{\Surf_\frc}\cap \partial M=\emptyset$.

\begin{figure}[h]
  \begin{center}
    \includegraphics[width=0.4\linewidth]{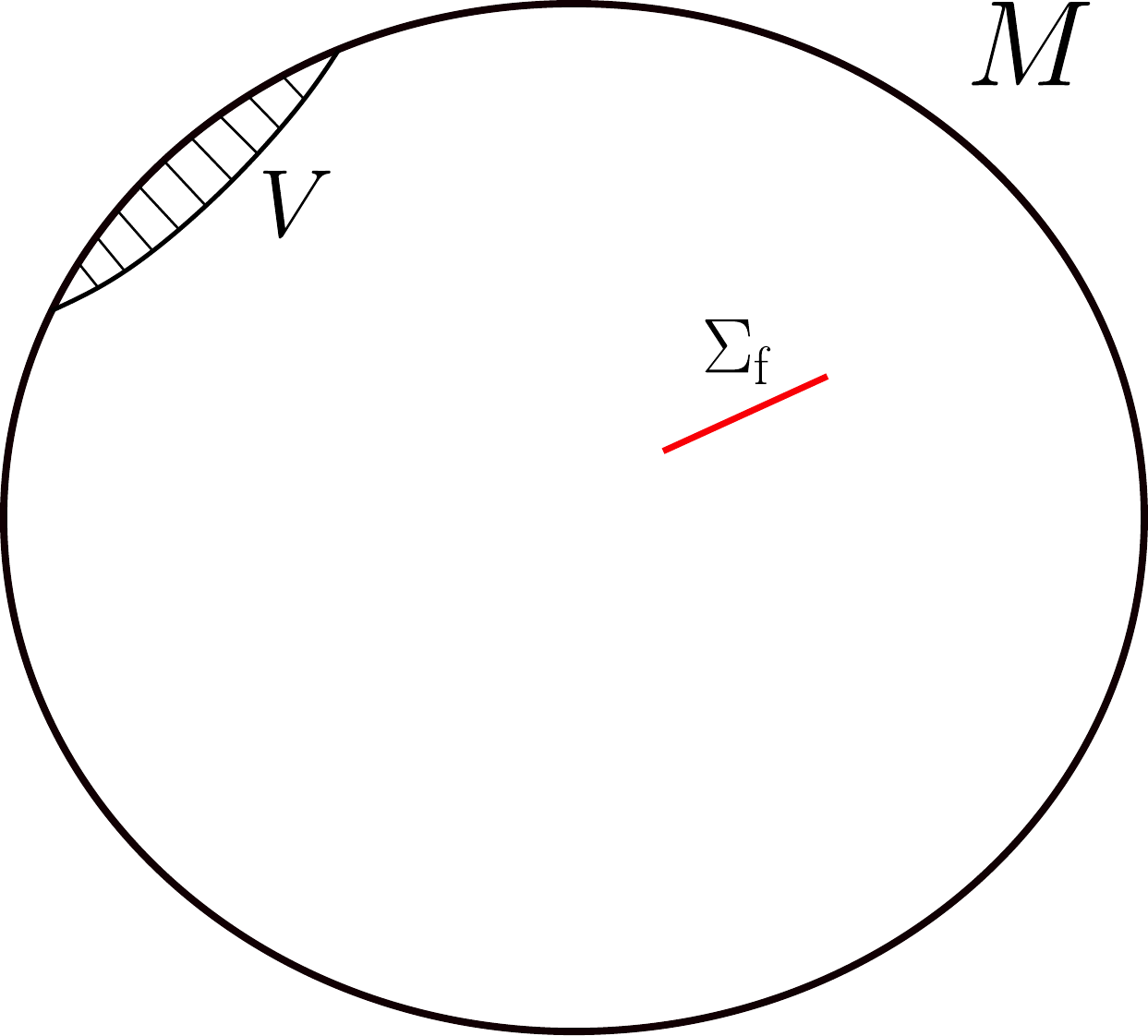}
    \caption{Rupture surface $\Surf_\frc$. The set $V$ is the observation domain. The figure is motivated by the \emph{Hi-net} seismograph network\protect\footnotemark in Japan. Hi-net seismic stations have seismometers installed in boreholes with varying depths. Note that the figure is not to scale.}
    \label{fig_intro}
  \end{center}
\end{figure}
\footnotetext{High Sensitivity Seismograph Network Japan ({\tt https://www.hinet.bosai.go.jp}).}

The seismic wave $\veu$ is modeled by the following equation of motion
\begin{equation} \label{eq-seismic-intro}
\rho \partial_t^2 \veu -\nabla\cdot(\tsLaINI:\nabla\veu)=0 \quad \textrm{in } M \setminus \overline{\Surf_\frc},
\end{equation}
where $\tsLaINI$ is the prestressed elasticity tensor.
In the case of isotropy and hydrostatic prestress $\tsT^0 = -p^0
\boldsymbol{I}$, the prestressed elasticity tensor $\tsLaINI$ has the form
\begin{equation} \label{T0-simple-intro}
   \itLaINI_{ijkl} = \lambda \delta_{ij} \delta_{kl}
   + \mu (\delta_{ik} \delta_{jl} + \delta_{il} \delta_{jk})
   - p^0 (\delta_{ij} \delta_{kl} - \delta_{il} \delta_{jk}) .
\end{equation}
In such case, the equation \eqref{eq-seismic-intro} has the form of the classical elasticity system which can be written as the system of hyperbolic equations \eqref{eq_system}. With our stability results on the unique continuation for the system \eqref{eq_system}, we can determine the displacement $\veu$ on both sides of the rupture surface, and the friction force $\vetau_\frc$, see Section \ref{section-fault}.

\smallskip
\noindent
{\bf Informal Formulation of Result.} {\it Let $M^3$ be the solid Earth with smooth boundary, and $\Surf_\frc$ be a smooth rupture surface.
We observe the seismic wave $\veu$ on the time interval $[-T,T]$ on an open subset $V\subset M\setminus \Surf_\frc$ (see Figure \ref{fig_intro}).
Then for sufficiently large $T$, we can determine the displacement $\veu$ on both sides of the rupture surface and the friction force $\vetau_\frc$, with explicit estimates in suitable norms.}

\smallskip
A precise formulation is given in Theorem \ref{determine-fault} and Corollary \ref{coro-higher-order}.
The tangential jump of particle velocity across the rupture surface
signifies the slip rate identified as a vector field. 
We write $\mathbi{s}$ for this quantity. In other words, 
$\mathbi{s}$ is the tangential component of 
    \begin{align*}
\jmp{\dot \veu} := \p_t (\veu_{+} - \veu_{-}),
\quad 
\veu_{\pm}:= \lim_{h\to 0^{\pm}} \veu (z+h\ven,t),\quad z\in \Surf_\frc,
    \end{align*}
where $\ven$ is a unit normal of the rupture surface.
The tangential component $\vetau_\frc$ of the (dynamic) traction $\vetau$
at the fault surface is the friction force.
The traction is  
    \begin{align*}
\vetau = \ven\cdot\tsT^0 +\vetau_1(\veu)+\vetau_2(\veu),
    \end{align*}
where $\vetau_1(\veu)$ and $\vetau_2(\veu)$ are defined by 
\eqref{eq:fric bc var} below, and $\ven\cdot\tsT^0$
is the contribution from
the known, static prestress $\tsT^0$. 
The normal component 
    \begin{align*}
\sigma_n = \ven \cdot \vetau
    \end{align*}
of this traction stands for the normal stress. The slip rate and normal stress are related to the
friction force through a friction law of the form
    \begin{align*}
\vetau_\frc = F(\mathbi{s}, \sigma_n).
    \end{align*}
It is typically assumed
that the friction force $\vetau_\frc$ and slip rate $\mathbi{s}$ are aligned, that is, parallel.
Several choices of $F$ have been introduced in the geophysics literature. 
Examples include the Slip Law and Aging Law in Rate- and State-dependent
Friction. It is common practice to invoke a simpler, linear
slip-weakening model to describe friction during a rupture when
afterslip is not considered, see e.g. \cite{DDLL}.

The friction law $F(\mathbi{s}, \sigma_n)$ in general is given in terms of a
few (presumably time-independent) parameter functions, and is typically a nonlinear integral operator.
In the geophysics literature this is expressed by introducing a state-variable function, see \cite{RLR} for the case of Rate- and State-dependent Friction.
The unique continuation provides the slip rate $\mathbi{s}$, normal stress $\sigma_n$ and friction force $\vetau_\frc$.
After fixing a parametric form of the friction law $F$, the inverse friction problem concerns the (conditional) recovery of the mentioned parameter functions in $F$. 
The inverse friction problem with one earthquake can be considered as a single measurement inverse problem for the parameters in $F$.
In the case of Rate- and State-dependent Friction, an ordinary differential equation determines a map from slip rate and normal stress to state-variable function that is also given in terms of a few parameter functions. We note that the regularity of
solutions restricts the allowable mapping property of the friction ``coefficient'' in the Amontons-Coulomb law that is widely applied.

We plan to analyze the inverse friction problem in a follow-up paper. To facilitate this, in view of nonlinearity of $F$, we give stability results for the unique continuation problem in Sobolev spaces that are Banach algebras, see Corollary \ref{coro-higher-order} and Remark \ref{rem_balg} below. 


\section{Unique continuation for system of hyperbolic equations}
\label{section-uc}

Consider the system of hyperbolic equations on $\R\times \R^n$,
\begin{equation}
P_i u_i+L_i (D u,u)=f_i,\quad  i=1,\cdots,m,
\end{equation}
where $P_i$ is the wave operator \eqref{def-wave} with time-independent wave speed $v_i$,
and $L_i$ are linear functions of $D u_k,u_k$ ($k=1,\cdots,m$)
with time-independent $L^{\infty}(\R^n)$ coefficients.
More precisely,
\begin{equation} \label{def-Li}
L_i(Du,u)=\sum_{k=1}^m \sum_{l=0}^n L_{i;kl} (D_l u_k) + L_{i;k}u_k,
\end{equation}
where $L_{i;kl},L_{i;k}\in L^{\infty}(\R^n)$.
We assume $v_i\in C^1(\R^n)$, $v_i>0$, $g^{jk}\in C^1(\R^n)$ and $h_j,q\in C^{0}(\R^n)$.

\smallskip
We will frequently use the following notations.
Denote
\begin{equation}
\|L_i\|_{\infty} := \max \Big\{ \max_{k,l} \|L_{i;kl}\|_{L^{\infty}(\R^{n})}, \max_{k} \|L_{i;k}\|_{L^{\infty}(\R^{n})} \Big\},\quad \|L\|_{\infty}:=\max_i \|L_i\|_{\infty}.
\end{equation}
We write 
$$u=(u_1,\cdots,u_m), \quad f=(f_1,\cdots,f_m).$$
Denote by $\|\cdot\|_0$ the $L^2(\R^{n+1})$-norm and by $\|\cdot\|_1$ the $H^1(\R^{n+1})$-norm.
Recall the weighted norm
\begin{equation}
\|u_i\|_{1,\tau}^2 :=\tau^2\|u_i\|_0^2+\|D u_i\|_0^2\,.
\end{equation}
We denote
$$\|u\|_0^2=\sum_{i=1}^m \|u_i\|_0^2, \quad \|u\|_{1,\tau}^2=\sum_{i=1}^m \|u_i\|_{1,\tau}^2.$$

Let $A(D_0)$ be the pseudo-differential operator with symbol $a(\xi_0)$, where $a\in C^{\infty}(\R)$ is a smooth function. It is formally defined as
$$A(D_0) w:= \mathcal{F}^{-1}_{\xi_0 \to  t} a(\xi_0) \mathcal{F}_{t'\to \xi_0} w,$$
where $\mathcal{F},\mathcal{F}^{-1}$ stand for the Fourier transform and its inverse. In particular, we consider the operator $e^{-\epsilon D_0^2/2\tau}$,
$$e^{-\epsilon D_0^2/2\tau} w:= \mathcal{F}^{-1}_{\xi_0 \to  t} e^{-\epsilon \xi_0^2/2\tau} \mathcal{F}_{t'\to \xi_0} w.$$
It can also be understood as an integral operator in the time variable with the kernel $(\tau/2\pi\epsilon)^{1/2} e^{-\tau |t'-t|^2/2\epsilon}$.

\smallskip
Let $\Gamma$ be a conical subspace of the cotangent bundle $T^{\ast}\R^{n+1}$, and let $\Gamma_{y_0}$ be its fiber at $y_0 \in \R^{n+1}$.
We recall the definition of a strongly pseudoconvex function in $\Gamma$ (Definition 2.4 in \cite{Tataru2}). A $C^2$ real-valued function $\phi$ is called \emph{strongly pseudoconvex in $\Gamma$} with respect to a partial differential operator $P$ at $y_0$ if
$$\textrm{Re}\{\overline{p},\{p,\phi\}\}(y_0,\xi)>0\quad \textrm{ on } p(y_0,\xi)=0,\; 0\neq \xi\in \Gamma_{y_0},$$
and
$$\frac{1}{i\tau}\{\overline{p(y,\xi+i\tau\phi'(y))}, p(y,\xi+i\tau\phi'(y))\}(y_0,\xi)>0,$$
for any $0\neq \xi\in \Gamma_{y_0}$ satisfying $p(y_0,\xi+i\tau \phi'(y_0))=0$, $\tau>0$. Here $p$ denotes the principle symbol of $P$, and $\{\cdot,\cdot\}$ denotes the Poisson bracket. 

When $P$ is a second order operator, the last condition above is void for non-characteristic functions with respect to $P$. In particular, we will later consider the following type of function
$$\psi(t,x)=\big(T-|x-z|\big)^2-t^2,$$
which is non-characteristic in $\big\{(t,x)\in \R\times \R^{n}: \psi(t,x) >0 \big\}$ with respect to the wave operator with constant wave speed $1$, where $T>0$ and $z\in \R^{n}$ are fixed.

\smallskip
In the coordinate $y=(t,x)\in \R\times \R^{n}$, the conormal bundle over $\R^{n+1}$ with respect to the foliation $x=\textrm{const}$ is 
\begin{equation*} \label{def-conormal}
N^{\ast} F:=\big\{(y,\xi)\in T^{\ast}\R^{n+1}: \xi=(\xi_0, \xi_1,\cdots,\xi_{n}), \,\xi_0=0 \big\}.
\end{equation*}
The conormal bundle over a subset $K\subset \R^{n+1}$ (with respect to the foliation $x=\textrm{const}$) is 
\begin{equation} \label{def-fiber}
N^{\ast}_K F:=\big\{(y,\xi)\in N^{\ast}F : y\in K \big\}.
\end{equation}

\smallskip
\subsection{Local estimates.}

Let us recall the following Carleman estimate in \cite{Tataru1}.

\begin{theorem}[Tataru]\label{th_carleman} 
Let $\Omega$ be an open subset of $\R\times\R^n$
and $P$ be the wave operator with time-independent coefficients.
Let $y_0 \in \Omega$ and $\psi \in C^{2,\rho}(\Omega)$ for some fixed $\rho\in(0,1)$, such that $\psi(y_0)=0$, $\psi'(y_0)\neq 0$ and $S=\{y\in \Omega:\psi(y)=0\}$ is an  oriented hypersurface non-characteristic at $y_0 \in S$. 

Then there exist $\kappa>0$ and a real-valued quadratic polynomial $\phi$, such that $\phi$ is strongly pseudoconvex in the conormal bundle over $B_{\kappa}(y_0)$ with respect to $P$, with the property that $\phi(y_0)=0$ and
\begin{equation}\label{phi-support-condition}
\{y: \psi(y) \leq 0\}\subset \{y: \phi(y)<0\}\cup \{y_0\} \quad \textrm{ in }B_{\kappa}(y_0).
\end{equation}

As a consequence, there exist constants $\epsilon_0,\tau_0, C, R$, such that for $\epsilon<\epsilon_0$ and $\tau > \tau_0$, we have
\begin{eqnarray*}
 \|e^{-\epsilon D_0^2/2\tau} e^{\tau \phi} u\|_{1,\tau} \le C \, \tau^{-1/2}\|e^{-\epsilon D_0^2/2\tau} e^{\tau \phi} Pu\|_{0} + C\, e^{-\tau \kappa^2/4\epsilon} \|e^{\tau \phi} u\|_{1,\tau},
\end{eqnarray*}
whenever $u \in H^1_{loc}(\Omega)$ satisfying $Pu \in L^2(\Omega)$ and $\supp(u) \subset B_{R}(y_0)$.
\end{theorem}

To begin with, we derive a Tataru-type estimate for the hyperbolic system \eqref{eq_system}.

\begin{proposition} \label{Carleman-system}
Let $\Omega$ be an open subset of $\R\times\R^n$.
Let $y_0 \in \Omega$ and $\psi \in C^{2,\rho}(\Omega)$ for some fixed $\rho\in(0,1)$, such that $\psi(y_0)=0$, $\psi'(y_0)\neq 0$ and $S=\{y\in \Omega:\psi(y)=0\}$ is an  oriented hypersurface non-characteristic at $y_0 \in S$. Suppose $u=(u_1,\cdots,u_m),\, u_i\in H^1(\Omega)$ is a solution of the hyperbolic system \eqref{eq_system} with $f_i\in L^2(\Omega)$.

Then there exist constants $\kappa,\epsilon_0,\tau_0, C, R$ and a real-valued quadratic polynomial $\phi$, as determined in Theorem \ref{th_carleman}, such that the following estimate holds for $\epsilon<\epsilon_0$ and $\tau > \tau_0$,
\begin{eqnarray*}
 \|e^{-\epsilon D_0^2/2\tau} e^{\tau \phi} u\|_{1,\tau} \le C \, \tau^{-1/2}\|e^{-\epsilon D_0^2/2\tau} e^{\tau \phi} f\|_{0} + C \, e^{-\tau \kappa^2/4\epsilon} \|e^{\tau \phi} u\|_{1,\tau},
\end{eqnarray*}
as long as $\supp(u) \subset B_{R}(y_0)$.
\end{proposition}

\begin{proof}
We apply Theorem \ref{th_carleman} to each component $u_i$ with the hyperbolic operator $P_i$,
\begin{eqnarray*}
 \|e^{-\epsilon D_0^2/2\tau} e^{\tau \phi} u_i\|_{1,\tau} \le C \, \tau^{-1/2}\|e^{-\epsilon D_0^2/2\tau} e^{\tau \phi} P_i u_i\|_{0} + C \, e^{-\tau \kappa^2/4\epsilon} \|e^{\tau \phi} u_i\|_{1,\tau}.
\end{eqnarray*}
We only need to estimate the first term on the right-hand side,
\begin{eqnarray*}
\|e^{-\epsilon D_0^2/2\tau} e^{\tau \phi} P_i u_i\|_{0} &\leq& \big\|e^{-\epsilon D_0^2/2\tau} e^{\tau \phi} \big(f_i-L_i(Du,u)\big) \big\|_{0} \\
&\le& \|e^{-\epsilon D_0^2/2\tau} e^{\tau \phi} f_i\|_{0} + \|L\|_{\infty}\sum_{k=1}^m  \|e^{-\epsilon D_0^2/2\tau} e^{\tau \phi} u_k\|_{0} \\
&& + \|L\|_{\infty}\sum_{k=1}^m  \|e^{-\epsilon D_0^2/2\tau} e^{\tau \phi} D u_k\|_{0}.
\end{eqnarray*}

For sufficiently large $\tau$, the second term on the right can be absorbed into the left-hand side when we sum over $i$. It suffices to estimate the last term,
\begin{eqnarray*}
\|e^{-\epsilon D_0^2/2\tau} e^{\tau \phi} D_l u_k\|_{0} &\leq& \|D_l \Big( e^{-\epsilon D_0^2/2\tau} e^{\tau \phi} u_k\Big)\|_{0}+ \|e^{-\epsilon D_0^2/2\tau} \big(D_l e^{\tau \phi}\big) u_k\|_{0} \\
&\leq& \| e^{-\epsilon D_0^2/2\tau} e^{\tau \phi}  u_k\|_{1}+  \|e^{-\epsilon D_0^2/2\tau} (\tau D_l \phi) e^{\tau \phi} u_k\|_{0},
\end{eqnarray*}
where we have used the fact that $[D_l, e^{-\epsilon D_0^2/2\tau}]=0$ for all $l=0,1,\cdots,n$.
Since $\phi$ is a quadratic polynomial, we know (see (2.303) in \cite{KKL})
$$[e^{-\epsilon D_0^2/2\tau},\tau D\phi(y)]=\epsilon \, {\rm Hess}_{\phi}(0)(D_0,0,\cdots,0) e^{-\epsilon D_0^2/2\tau}, $$
which gives 
\begin{eqnarray*}
\|e^{-\epsilon D_0^2/2\tau} e^{\tau \phi} D_l u_k \|_{0} &\le& \| e^{-\epsilon D_0^2/2\tau} e^{\tau \phi} u_k\|_{1}+  \tau \|\phi\|_{C^1} \|e^{-\epsilon D_0^2/2\tau}  e^{\tau \phi} u_k\|_{0} \\
&& +  \epsilon \|\phi\|_{C^2} \| e^{-\epsilon D_0^2/2\tau} e^{\tau \phi} u_k\|_{0}.
\end{eqnarray*}

Combining the estimates above and summing the inequalities over $i=1,\cdots,m$, for sufficiently large $\tau$ (and $\epsilon<1$), one can absorb all unwanted terms on the right-hand side into the left-hand side. The proposition is proved.
\end{proof}

Next, we use Proposition \ref{Carleman-system} to derive a local stability estimate for the lower temporal frequencies for the hyperbolic system \eqref{eq_system}, similar to Theorem 1.1 in \cite{BKL1}.

\begin{figure}[h]
  \begin{center}
    \includegraphics[width=0.6\linewidth]{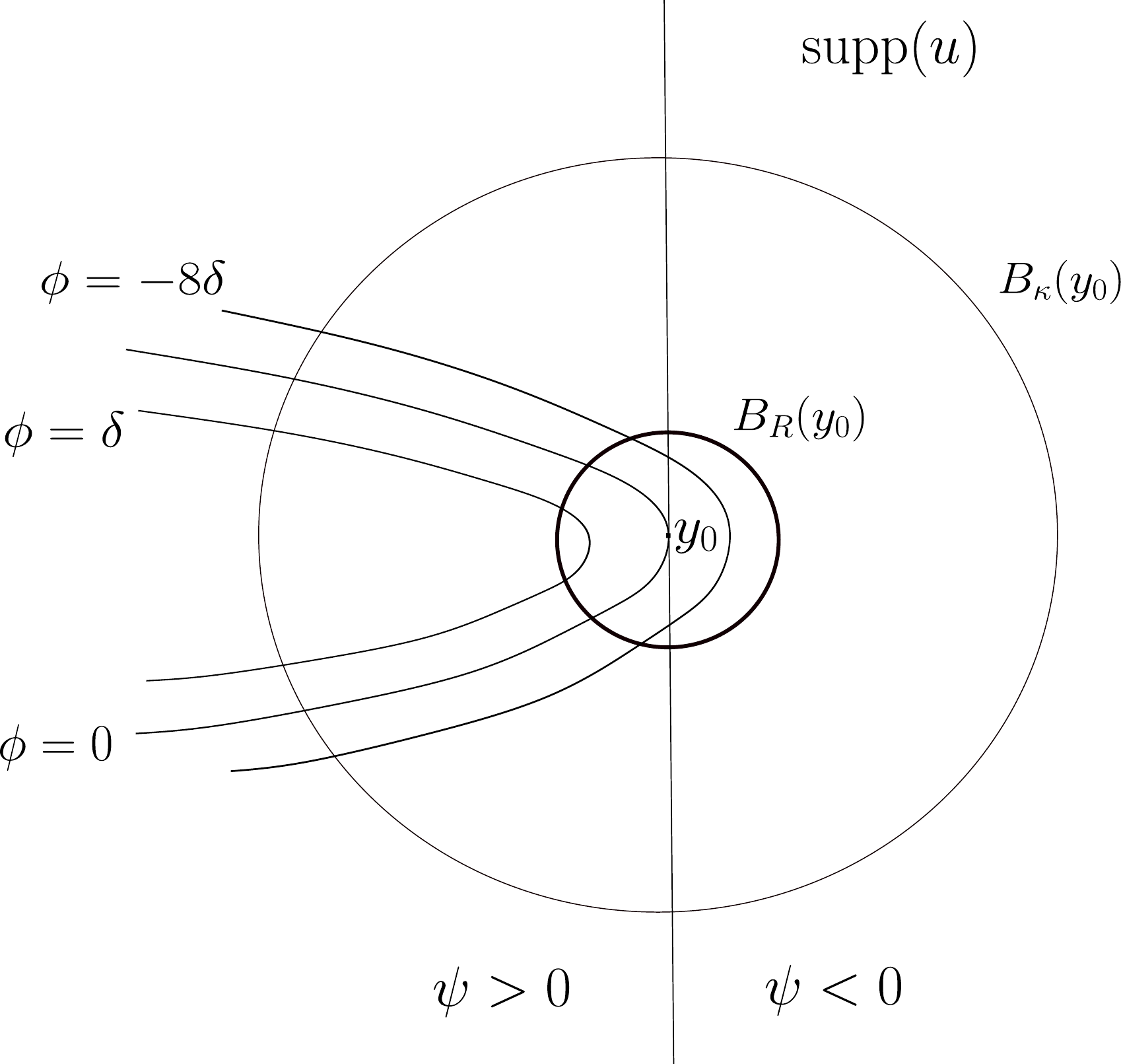}
    \caption{Setting of Lemma \ref{lemma_exp1}.}
    \label{fig_local}
  \end{center}
\end{figure}

\begin{notation*}
Let $b\in C_0^{\infty}(\R^{n+1})$, $0\leq b(\xi) \leq 1$ be supported in $|\xi|\leq 2$ and equal to 1 in $|\xi|\leq 1$.
Let $A(D_0)$ be a pseudo-differential operator with symbol $a\in C_0^{\infty}(\R)$, $0\leq a \leq 1$, where $a$ is supported in $[-2,2]$ and equal to 1 in $[-1,1]$. 

Let $\psi,\phi,y_0,\kappa,R$, ($R<\kappa/2$) be as stated in Theorem \ref{th_carleman}.
As Proposition 2.5 in \cite{BKL1}, one can choose $\delta>0$ sufficiently small such that (see Figure \ref{fig_local})
\begin{equation}\label{phi-condition-delta}
\{y: \psi(y)\leq 0\} \cap \{y\in B_{\kappa}(y_0): \phi(y) \geq -8\delta\} \subset B_R(y_0).
\end{equation}
Let $\chi\in C_0^{\infty}(\R)$, $0\leq \chi\leq 1$ be a localizer be supported in $[-8\delta,\delta]$ and equal to 1 in $[-7\delta,\delta/2]$. In particular, we choose the functions $b,a,\chi$ from the Gevrey functions of class $1/\alpha$ for a fixed $\alpha\in (0,1)$.
\end{notation*}

\begin{lemma}\label{lemma_exp1}
Let $\Omega$ be an open subset of $\R\times\R^n$.
Let $S=\{y\in \Omega:\psi(y)=0\}$ be an oriented $C^{2,\rho}$ hypersurface which is non-characteristic in $\Omega$, and $y_0\in S$, $\psi'(y_0)\neq 0$. 
Let $b\in C_0^{\infty}(\R^{n+1})$ be a Gevrey localizer of class $1/\alpha$ for a fixed $\alpha\in (0,1)$ as defined above.
Suppose $u=(u_1,\cdots,u_m),\, u_i\in H^1(\Omega)$ is a solution of the hyperbolic system \eqref{eq_system} with $f_i\in L^2(\Omega)$.
Assume $u$ satisfies 
\begin{equation}\label{psi-support}
\supp(u) \subset \{y: \psi(y)\leq 0\} \cap \Omega.
\end{equation}
Then there exist constants $R,r,c_0,c_1>0$ such that the following holds.

For $\mu \geq 1$, if for some constant $C_0>0$,
\begin{eqnarray*}
\|u\|_{H^1(B_{2R}(y_0))} \leq C_0,\quad \|f\|_{L^2(B_{2R}(y_0))} \leq C_0, \quad \big\|A(\frac{D_0}{\mu})b(\frac{y-y_0}{R})f \big\|_{L^2} \leq C_0 \exp(-\mu^\alpha),
\end{eqnarray*}
then there exists a constant $C_1>0$ independent of $\mu$ such that
\begin{eqnarray*}
\big\|A(\frac{D_0}{\omega})b(\frac{y-y_0}{r})u \big\|_{H^1} \le C_1 \exp(-c_1\mu^{\alpha^2}),\quad \forall\, \omega \leq \mu^\alpha/c_0.
\end{eqnarray*}
Here $c_0,c_1$ are independent of $\mu,C_0$.
\end{lemma}

\begin{proof}
We follow the proof of Theorem 1.1 in \cite{BKL1}.
Denote 
\begin{equation}\label{def-b0}
b_0 :=b(\frac{y-y_0}{R}).
\end{equation}
Consider the functions 
\begin{equation}
\overline{u}_i:=\chi(\phi) b_0 u_i, \quad \overline{u}=(\overline{u}_1,\cdots,\overline{u}_m).
\end{equation}
where $\phi$ is the quadratic polynomial determined in Proposition \ref{Carleman-system}. The function $\overline{u}$ is supported in $B_{2R}(y_0)$, and satisfies
\begin{eqnarray*}
P_i \overline{u}_i &=& \chi(\phi) b_0 P_i u_i + [P_i, \chi(\phi) b_0] u_i \\
&=& \chi(\phi) b_0 f_i - \chi(\phi) b_0 L_i(Du,u) + [P_i, \chi(\phi) b_0] u_i.
\end{eqnarray*}
Since $L_i$ is a linear operator, by \eqref{def-Li},
\begin{eqnarray*}
L_i ( D \overline{u}, \overline{u} ) = \chi(\phi) b_0 L_i(Du,u) + \sum_{k=1}^m \sum_{l=0}^n L_{i; kl} D_l \big(\chi(\phi) b_0 \big) u_k.
\end{eqnarray*}
Hence $\overline{u}$ satisfies the equations
\begin{equation} \label{eq-local-u}
P_i \overline{u}_i + L_i ( D \overline{u}, \overline{u} ) = \chi(\phi) b_0 f_i + [P_i, \chi(\phi) b_0] u_i+ \sum_{k=1}^m \sum_{l=0}^n L_{i; kl} D_l \big(\chi(\phi) b_0 \big) u_k.
\end{equation}
Then we apply Proposition \ref{Carleman-system} to $\overline{u}=\chi(\phi) b_0 u_i$ for the hyperbolic system \eqref{eq-local-u},
\begin{eqnarray} \label{Carleman-local}
C^{-1} \, \tau^{1/2} \|e^{-\epsilon D_0^2/2\tau} e^{\tau \phi} \overline{u}\|_{1,\tau} \leq \|e^{-\epsilon D_0^2/2\tau} e^{\tau \phi} \chi(\phi) b_0 f\|_{0} + \sum_i \big\|e^{-\epsilon D_0^2/2\tau} e^{\tau \phi} [P_i, \chi(\phi) b_0] u_i \big\|_{0} \nonumber \\
+ \|L\|_{\infty} \sum_{k,l} \big\|e^{-\epsilon D_0^2/2\tau} e^{\tau \phi} D_l \big(\chi(\phi) b_0 \big) u_k \big\|_{0}
+ \tau^{1/2} e^{-\tau \kappa^2/4\epsilon} \|e^{\tau \phi} \overline{u} \|_{1,\tau}.
\end{eqnarray}

By definition \eqref{def-b0}, we see that $\supp(b_0)\subset B_{2R}(y_0) \subset B_{\kappa}(y_0)$, in view of $R<\kappa/2$.
Thus the conditions \eqref{phi-condition-delta} and \eqref{psi-support} imply 
\begin{equation} \label{support-in-R}
\supp(u)\cap \supp(\chi (\phi)) \cap B_{2R}(y_0)\subset B_R(y_0).
\end{equation}
Since $b_0=1$ in $B_R(y_0)$, we have $[P_i,\chi(\phi) b_0] u_i=[P_i,\chi(\phi)] u_i$ in $B_{2R}(y_0)$, and the following holds everywhere:
\begin{equation} \label{commutator-b0}
[P_i,\chi(\phi) b_0] u_i=[P_i,\chi(\phi)]b_0 u_i.
\end{equation}
Moreover, the conditions \eqref{phi-support-condition} and \eqref{psi-support} imply
\begin{equation} \label{support-in-delta}
\supp(\chi'(\phi))\cap \supp(b_0 u) \subset \{y: -8\delta \leq \phi(y)\leq -7\delta\},
\end{equation}
where $\chi'(\phi)$ denotes the derivative of $\chi(\phi)$.

The second term on the right-hand side of \eqref{Carleman-local} can be estimated using \eqref{commutator-b0} and \eqref{support-in-delta} as follows,
\begin{eqnarray} \label{Carleman-local-1}
\big\|e^{-\epsilon D_0^2/2\tau} e^{\tau \phi} [P_i,\chi(\phi) b_0] u_i \big\|_{0} &=& \big\|e^{-\epsilon D_0^2/2\tau} e^{\tau \phi} [P_i,\chi(\phi)] b_0 u_i \big\|_{0} \nonumber \\
&\leq& C e^{-7\tau\delta} \|u_i\|_{H^1(B_{2R}(y_0))}.
\end{eqnarray}
We note that the constant $C$ here also depends on $\delta,R$.
For the third term on the right-hand side of \eqref{Carleman-local}, notice that $\chi(\phi)(D_l b_0) u_k=0$ due to \eqref{support-in-R}. Hence by \eqref{support-in-delta},
\begin{eqnarray}
\big\|e^{-\epsilon D_0^2/2\tau} e^{\tau \phi} D \big(\chi(\phi) b_0 \big) u_i \big\|_{0} &=& \|e^{-\epsilon D_0^2/2\tau} e^{\tau \phi} \chi'(\phi) b_0  u_i\|_{0}  \nonumber \\
&\leq& C e^{-7\tau\delta} \|u_i\|_{L^2(B_{2R}(y_0))}.
\end{eqnarray}
For the fourth term on the right, note that the parameter $\delta$ can be chosen such that $8\delta< \kappa^2/4\epsilon$. Since $\phi\leq \delta$ in the support of $\chi(\phi)$, the last term is bounded by 
\begin{equation}
\tau^{1/2} e^{-\tau \kappa^2/4\epsilon} \|e^{\tau \phi} \overline{u} \|_{1,\tau} \leq \tau^{3/2} e^{-8\tau\delta} e^{\tau\delta} \|\overline{u}\|_{H^1} \leq C e^{-6\tau\delta}.
\end{equation}
The first term on the right-hand side of \eqref{Carleman-local} can be estimated by repeating the proof of Lemma 2.6 in \cite{BKL1}.
This show that there is $c > 0$ such that 
\begin{equation} \label{Carleman-local-4}
\|e^{-\epsilon D_0^2/2\tau} e^{\tau \phi} \chi(\phi) b_0 f\|_{0} \leq Ce^{2\tau\delta-c\mu^{\alpha}},
\end{equation}
Note that here we replaced $f$ with $b_0 f$
using \eqref{support-in-R} and the fact that $\supp(f)\subset \supp(u)$.
Combining the estimates \eqref{Carleman-local-1}-\eqref{Carleman-local-4}, we obtain
\begin{equation} \label{Carleman-local-final}
\tau^{1/2} \|e^{-\epsilon D_0^2/2\tau} e^{\tau \phi} \chi(\phi) b_0 u\|_{1,\tau} \leq Ce^{2\tau\delta} (e^{-c\mu^{\alpha}}+e^{-8\tau\delta}),
\end{equation}
for sufficiently large $\tau>\tau_0$.

\smallskip
The estimate \eqref{Carleman-local-final} has the same form as (2.7) in \cite{BKL1}. Then one can follow the rest of the proof there. Here, we sketch the outline of the proof. The first part is to extend the estimate \eqref{Carleman-local-final} to the upper complex plane. More precisely, consider 
\begin{equation}
N(\tau):= \|e^{-\epsilon D_0^2/2\tau} e^{\tau \phi} \chi(\phi) b_0 u\|_{H^1}^2,\quad \tau\in \mathbb{C}.
\end{equation}
One needs to show that there is $\tilde c > 0$ so that in the region 
    \begin{align*}
\mathcal R(\mu) := \{z \in \C : |z|\leq \tilde c\mu^{\alpha},\ \textrm{Im}z\geq 0\},
    \end{align*} 
the following holds:
\begin{equation} \label{estimate-Nz}
N(-iz)\leq C(1+|z|^2)e^{-10\delta \textrm{Im}z}.
\end{equation}
There is $\tilde c > 0$ such that
the inequality \eqref{estimate-Nz} is true for $z=i\tau \in \mathcal R(\mu)$ when $\tau>\tau_0$, as follows from \eqref{Carleman-local-final}. The estimate can be immediately extended to $z=i\tau \in \mathcal R(\mu)$, $0 \leq \tau\leq\tau_0$, as $\tau_0$ can be put into the constant $C$. To extend the estimate to the whole upper complex plane, one needs a complex analysis argument using the Phragmen-Lindel\"of principle, see Lemma 2.7 in \cite{BKL1}. 

The second part is to estimate 
\begin{equation}
F(y):= A(\frac{\beta D_0}{\mu^{\alpha}}) \big( \eta(\phi)u \big)(y),
\end{equation}
with $\beta>0$ to be determined.
Here $\eta(s):=\eta_1(s/\delta)$, where $\eta_1$ of Gevrey class $1/\alpha$ is a localizer supported in $[-4,1]$ and equal to 1 in $[-3,1/2]$. By our construction, $\chi(\phi)=1$ on $\supp\big(\eta(\phi)u \big)$, and hence $\eta(\phi)u=\eta(\phi) \chi(\phi)u$. The function $F$ can be written as an integral over $\R$:
\begin{equation}
F(y)=\int_{\mathbb{R}} \bar{\widehat{\eta}}(\bar{z})\Big( A(\frac{\beta D_0}{\mu^{\alpha}}) e^{-iz\phi} \chi(\phi) u \Big)(y) dz,
\end{equation} 
where $\widehat{\eta}$ denotes the Fourier transform of $\eta$.
Then we change the integral over $\mathbb{R}$ to a contour integral in the complex plane
    \begin{align*}
F = I_1 + I_2, \quad I_j = \int_{\Gamma_j} \bar{\widehat{\eta}}(\bar{z}) A(\frac{\beta D_0}{\mu^{\alpha}}) e^{-iz\phi} \chi(\phi) u\, dz, \quad j=1,2,
    \end{align*}
where, writing $\ell = \frac{1}{\sqrt 2} \tilde c \mu^\alpha$,
$\Gamma_1 = \{ z \in \R : |z| \geq \ell \}$ and 
    \begin{align*}
\Gamma_2 = \{ z \in \C : \textrm{Re} z = \pm \ell,\ 0 \leq \textrm{Im} z \leq \ell \} \cup \{ z \in \C : |\textrm{Re} z| \leq \ell,\ \textrm{Im} z = \ell\}.
    \end{align*}
Note that $\Gamma_2 \subset \mathcal R(\mu)$.
As $\eta_1$ is of Gevrey class $1/\alpha$ and supported in $[-4,1]$, there holds
    \begin{align}\label{eta_gevrey}
|\bar{\widehat{\eta}}(\bar{z})| 
\leq 
C e^{4 \delta \textrm{Im} z - c |\textrm{Re} z|^\alpha}
    \end{align}
for some $c > 0$. It follows that 
$\|I_1\|_{H^1} \le C e^{-c \mu^{\alpha^2}}$
for some $c > 0$. Moreover, 
    \begin{align*}
\|I_2\|_{H^1}
\leq 
C \int_{\Gamma_2} e^{4 \delta \textrm{Im} z - c |\textrm{Re} z|^\alpha}
\|A(\frac{\beta D_0}{\mu^{\alpha}}) e^{-\epsilon D_0^2/2i z} \|_{B(H^1)} 
\| e^{-\epsilon D_0^2/(-2i z)} e^{-iz\phi} \chi(\phi) u \|_{H^1} |dz|,
    \end{align*}
where $B(H^1)$ is the space of bounded operators on $H^1$.
Now we apply (\ref{estimate-Nz}) together with the estimate
    \begin{align*}
\|A(\frac{\beta D_0}{\mu^{\alpha}}) e^{-\epsilon D_0^2/2i z} \|_{B(H^1)} \leq e^{c \textrm{Im} z/ \beta^2},
    \end{align*}
that holds for some $c > 0$. This yields
    \begin{align*}
\|I_2\|_{H^1}
\leq
C \int_{\Gamma_2} 
e^{4 \delta \textrm{Im} z - c |\textrm{Re} z|^\alpha
+ c \textrm{Im} z/ \beta^2 - 5 \delta \textrm{Im} z}
|dz|
\leq 
C e^{-c \mu^{\alpha^2}}
    \end{align*}
for some $c > 0$ and large enough $\beta > 0$.

By combining the above estimates for $I_1$ and $I_2$ 
we get $\|F\|_{H^1} \le C e^{-c \mu^{\alpha^2}}$
for large $\beta$, which is very close to the claimed estimate.
The final step of the proof is to replace the cut off function $\eta(\phi)$
with the cut off function $b((y-y_0)/r)$.
We omit the details of the short proof and refer to \cite{BKL1},
see the end of the proof of Theorem 1.1 there.
\end{proof}

\subsection{Global estimates.}
Propagating the local estimate Lemma \ref{lemma_exp1} yields a global estimate.
To do this, we use the same constructions as Assumption A4 in \cite{BKL2}. 

\begin{assumption}\label{assumption1}
Let $\Omega$ be a bounded connected open subset of $\R\times\R^n$ and $u=(u_1,\cdots,u_m)$, $u_i \in H^1(\Omega)$.
Assume that there is a function $\psi \in C^{2,\rho}(\Omega)$ for some  $\rho\in(0,1]$,  such that in an open set $\Omega_0 \subset \Omega$ one has $\psi'(y)\neq 0$ and $p_i(y,\psi'(y))\neq 0$ for all $i$ and all $y\in \Omega_0$, where 
$p_i(y,\xi)= \xi_0^2 - v_i(x)^2 \sum_{j,k} g^{jk}(x) \xi_j\xi_k$ is the principal symbol of the hyperbolic operator $P_i$ in \eqref{def-wave}.

Assume that there exist values $\psi_{min} < \psi_{max}$ and a connected nonempty set $\Upsilon \subset \Omega_0$ such that:
$\supp(u)\cap \Upsilon = \emptyset$,
and $\emptyset \neq  \{y\in \Omega_0: \psi(y) >  \psi_{max}\} \subset \Upsilon$.
Assume that $\psi_{min}$ is such that the open set 
$\Omega_a = \{y \in \Omega_0 - \overline{\Upsilon}: \psi_{min} < \psi(y) < \psi_{max} \} $ is nonempty, connected and satisfies
$\dist(\partial \Omega_0, \Omega_a) >0$.
\end{assumption}

\begin{construction*}
Let $R>r>0$ satisfying $2R<\dist(\partial \Omega_0, \Omega_a)$. 
Under Assumption \ref{assumption1}, we choose a maximal $r/2$-separated set in $\overline{\Omega}_a$ as follows. Take $y_1$ to be a point where $\psi$ achieves maximum in $\overline{\Omega}_a$, and $y_2$ to be a point where $\psi$ achieves maximum in $\overline{\Omega}_a-B_{r/2}(y_1)$. In general, let $y_j\in \overline{\Omega}_a$ be a point where $\psi$ achieves maximum in $\overline{\Omega}_a-\cup_{l=1}^{j-1} B_{r/2}(y_l)$.
Observe that $y_1\in \partial \Upsilon$ and $y_j$ is on the boundary of $\cup_{l=1}^{j-1} B_{r/2}(y_l) \cup \Upsilon$, since $\psi$ has no critical point in $\Omega_0$ by assumption. See Figure \ref{fig_iteration}.
Repeat the procedure until it stops, and we get a maximal $r/2$-separated set $\{y_j\}_{j=1}^N$, also an $r/2$-net, in $\overline{\Omega}_a$. Since $\{y_j\}_{j=1}^N$ is $r/2$-separated, the total number $N$ of points is bounded by
\begin{equation} \label{bound-N}
N\leq C(n)\, {\rm vol}(\Omega_0) r^{-(n+1)}.
\end{equation}

Denote $u_{k,1}=u_k$ for $k=1,\cdots,m$, and define
\begin{equation} \label{def-uij}
u_{k,j+1}:=(1-b_j)u_{k,j},\quad  b_j:=b \big(\frac{2(y-y_j)}{r} \big), \quad j\geq 1,
\end{equation}
where $b\in C_0^{\infty}(\R^{n+1})$, $0\leq b(\xi) \leq 1$ is a localizer supported in $|\xi|\leq 2$ and equal to 1 in $|\xi|\leq 1$.
We write $u_{\cdot,j}:=(u_{1,j},\cdots,u_{m,j})$.
The requirement $2R<\dist(\partial \Omega_0, \Omega_a)$ yields $\cup_{j=1}^N B_{2R}(y_j)\subset \Omega_0$. 
\end{construction*}

\begin{figure}[h]
  \begin{center}
    \includegraphics[width=0.6\linewidth]{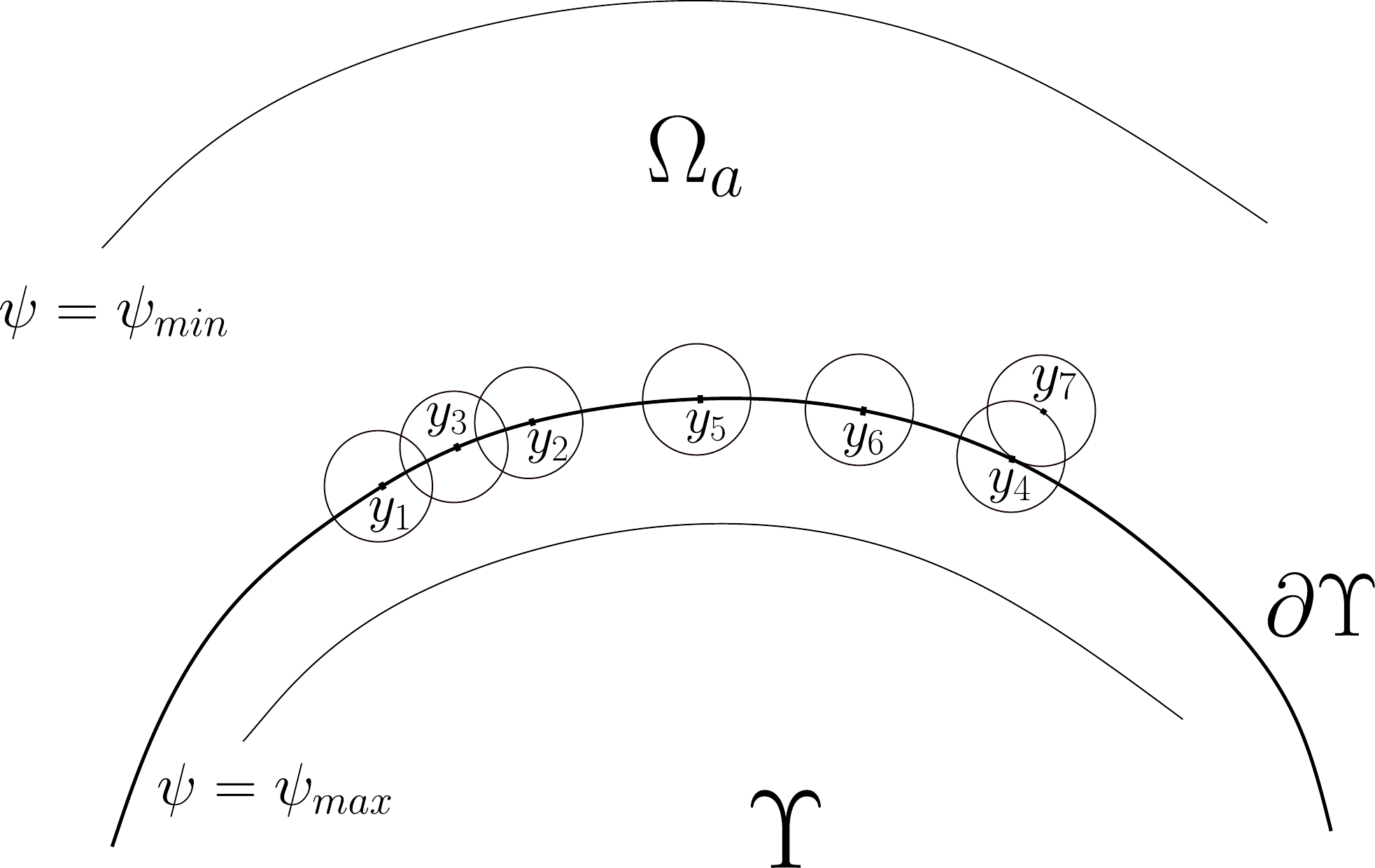}
    \caption{Choice of points $y_j$.}
    \label{fig_iteration}
  \end{center}
\end{figure}

The next lemma is an iteration of Lemma \ref{lemma_exp1}, analogous to Theorem 2.7 in \cite{BKL2}.

\begin{lemma}\label{iteration}
Under Assumption \ref{assumption1}, the points $y_j$ and functions $u_{k,j}$ $(j=1,\cdots,N)$ are defined as above.
Let $b\in C_0^{\infty}(\R^{n+1})$ be a Gevrey localizer of class $1/\alpha$ for a fixed $\alpha\in (0,1)$, as defined in Lemma \ref{lemma_exp1}.
Suppose $u=(u_1,\cdots,u_m),\, u_i\in H^1(\Omega)$ is a solution of the hyperbolic system \eqref{eq_system} with $f_i\in L^2(\Omega)$.
Then there exist constants $r,c_0,c_1,c_2,c_3,C_j$
such that the following holds. 

If for some $\mu>c_2$,
\begin{eqnarray*}\label{ipo1}
\|u\|_{H^1(\Omega_0)}= 1 ,\quad \|f\|_{L^2(\Omega_0)} \le  \exp(-\mu^{\alpha}),
\end{eqnarray*}
then we can find  $\mu_1=\mu$, $\mu_{j+1} = \mu_{j}^{\alpha}/c_3$, $\mu_j > 1$, such that $\|u_{k,j}\|_{H^1(\Omega_0)} \leq C(N,r)$ for $k=1,\cdots,m$, and
\begin{eqnarray*}\label{ipo4}
\big\| A(\frac{D_0}{\omega})b(\frac{y-y_j}{r})u_{k,j} \big\|_{H^1} \le C_j\exp({-c_1 \mu_j^{\alpha^2}}),\quad \forall\, \omega \le \mu_j^\alpha/c_0.
\end{eqnarray*}
\end{lemma}

\begin{proof}
For $j=1$, $u_{\cdot,1}=u$ satisfies the support condition \eqref{psi-support} in $\Omega_0$ for $\tilde{\psi}=\psi-\psi(y_1)$ due to Assumption \ref{assumption1}. This can be argued as follows. For $y\in \Omega_0$, $u(y)\neq 0$ shows $y\in \Omega_0 \setminus \Upsilon$ by Assumption \ref{assumption1}, which means either $y\in (\Omega_0 \setminus \Upsilon) \cap \overline{\Omega}_a$ or $y\in (\Omega_0 \setminus \Upsilon) \setminus \overline{\Omega}_a$. The former implies $\psi(y)\leq \psi(y_1)$ by our choice of $y_1$. The latter implies either $\psi(y) <\psi_{min} \leq \psi(y_1)$ or $\psi(y) >\psi_{max}$. However, the case of $\psi(y) >\psi_{max}$ does not happen due to Assumption \ref{assumption1}:
$$\{y \in \Omega_0: \psi(y) > \psi_{max}\}\subset \Upsilon \subset \{y\in \Omega_0: u(y)=0\}.$$ 

Let $R,r$ be the constants determined in Lemma \ref{lemma_exp1} .
Since $\|A(\frac{D_0}{\mu_1})b(\frac{y-y_1}{R})f\|_0\leq \|f\|_0\leq \exp(-\mu_1^{\alpha})$, applying Lemma \ref{lemma_exp1} gives
\begin{equation}\label{iteration_j=1}
\big\|A(\frac{D_0}{\omega})b(\frac{y-y_1}{r})u \big\|_{H^1} \leq C_1 e^{-c_1\mu_1^{\alpha^2}}, \quad  \forall \omega\leq \mu_1^{\alpha}/c_0.
\end{equation}

\smallskip
For $j=2$, $u_{\cdot,2}=(1-b_1)u$ satisfies the following equation
\begin{equation}\label{formula-u2}
P_i u_{i,2}+ L_i(D u_{\cdot,2},u_{\cdot,2})=(1-b_1)f_i -[P_i,b_1]u_i -\sum_{k=1}^m \sum_{l=0}^n L_{i;kl} (D_l b_1) u_k.
\end{equation}
This can be seen from \eqref{eq-local-u} by replacing $\chi,b_0$ with $1, 1-b_1$.
Observe that $D_l b_1$ is supported in $B_r(y_1)-B_{r/2}(y_1)$ where $b(\frac{y-y_1}{r})=1$. Then by \eqref{iteration_j=1} and Lemma 2.3(b) in \cite{BKL1}, for $\mu_2\leq \mu_1^{\alpha}/c_0$ and $\beta\geq 3$ to be determined later,
\begin{eqnarray} \label{estimate-Dbu}
\big\|A(\frac{\beta D_0}{\mu_2}) b(\frac{y-y_2}{R}) (D_l b_1) u_{i} \big\|_0&=& \big\|A(\frac{\beta D_0}{\mu_2}) b(\frac{y-y_2}{R}) (D_l b_1) b(\frac{y-y_1}{r})u_{i} \big\|_0 \nonumber \\
&\leq& Cr^{-1} \big\|A(\frac{D_0}{\mu_2}) b(\frac{y-y_1}{r}) u_{i} \big\|_0 + C e^{-\cbeta\mu_2^{\alpha}}\|u_i\|_0 \nonumber \\
&\leq& Cr^{-1} e^{-c_1 \mu_1^{\alpha^2}}+ Ce^{-\cbeta \mu_2^{\alpha}}.
\end{eqnarray}
Similarly,
\begin{equation} \label{estimate-DDbu}
\big\|A(\frac{\beta D_0}{\mu_2}) b(\frac{y-y_2}{R}) (D_s D_l b_1) u_{i} \big\|_0 \leq Cr^{-2} e^{-c_1 \mu_1^{\alpha^2}}+ Ce^{-\cbeta \mu_2^{\alpha}}.
\end{equation}
Since the coefficients $L_{i;kl}$ in \eqref{formula-u2} are time-independent, we have
\begin{equation} \label{estimate-LDbu}
\big\|A(\frac{\beta D_0}{\mu_2}) b(\frac{y-y_2}{R}) L_{i;kl} (D_l b_1) u_{k} \big\|_0 \leq C \|L\|_{\infty} r^{-1} e^{-c_1 \mu_1^{\alpha^2}}+ C \|L\|_{\infty} e^{-\cbeta \mu_2^{\alpha}}.
\end{equation}
Recall from Lemma 2.3 in \cite{BKL1} that
\begin{equation}
c_4=c_4(\beta)=C(\alpha, {\rm vol}(B_R))(1-\frac{2}{\beta})^{\alpha} \geq C(\alpha,n,R) \, 3^{-\alpha},
\end{equation}
since $\beta\geq 3$. Thus we can treat $c_4$ as a constant independent of $\beta$.

Next we estimate terms involving $Du$.
\begin{eqnarray*}
\big\|A(\frac{\beta D_0}{\mu_2}) b(\frac{y-y_2}{R}) (D_l b_1) (D_s u_{i}) \big\|_0 &\leq& \big\|D_s \Big( A(\frac{\beta D_0}{\mu_2}) b(\frac{y-y_2}{R}) (D_l b_1) u_{i} \Big) \big\|_0 \\
&&+ \big\|A(\frac{\beta D_0}{\mu_2}) D_s\Big(b(\frac{y-y_2}{R}) (D_l b_1) \Big) u_{i} \big\|_0 \\
&\leq&  \big\|A(\frac{\beta D_0}{\mu_2}) b(\frac{y-y_2}{R}) (D_l b_1) b(\frac{y-y_1}{r})  u_{i} \big\|_{H^1}. \\
&& + \big\|A(\frac{\beta D_0}{\mu_2}) D_s\Big(b(\frac{y-y_2}{R}) (D_l b_1) \Big) b(\frac{y-y_1}{r})  u_{i} \big\|_0,
\end{eqnarray*}
where we have used the fact that $[A(D_0),D_s]=0$, for all $s=0,1,\cdots,n$.
For the second term on the right-hand side, we apply Lemma 2.3(b) in \cite{BKL1} just as before. The first term on the right can be estimated as follows.
\begin{eqnarray*}
\big\|A(\frac{\beta D_0}{\mu_2}) b(\frac{y-y_2}{R}) (D_l b_1) b(\frac{y-y_1}{r})  u_{i} \big\|_{H^1} = \big\|A(\frac{\beta D_0}{\mu_2}) b(\frac{y-y_2}{R}) (D_l b_1) A(\frac{D_0}{\mu_2}) b(\frac{y-y_1}{r})  u_{i} \big\|_{H^1} \\
+ \big\|A(\frac{\beta D_0}{\mu_2}) b(\frac{y-y_2}{R}) (D_l b_1) \Big(1-A(\frac{D_0}{\mu_2})\Big) b(\frac{y-y_1}{r})  u_{i} \big\|_{H^1} \\
\leq Cr^{-2} \big\|A(\frac{D_0}{\mu_2}) b(\frac{y-y_1}{r}) u_{i} \big\|_{H^1}+ Cr^{-1} e^{-\cbeta \mu_2^{\alpha}} \|u_i\|_{H^1},
\end{eqnarray*}
where we have used Lemma 2.3(c) in \cite{BKL1} in the last inequality. Thus by \eqref{iteration_j=1}, we obtain
\begin{equation} \label{estimate-DbDu}
\big\|A(\frac{\beta D_0}{\mu_2}) b(\frac{y-y_2}{R}) (D_l b_1) (D_s u_{i}) \big\|_0 \leq Cr^{-2} e^{-c_1 \mu_1^{\alpha^2}}+Cr^{-1} e^{-\cbeta\mu_2^{\alpha}}.
\end{equation}
Combining \eqref{estimate-DDbu} and \eqref{estimate-DbDu} yields
\begin{eqnarray} \label{estimate-Pib1}
\big\|A(\frac{\beta D_0}{\mu_2}) b(\frac{y-y_2}{R}) [P_i,b_1] u_{i} \big\|_0 \leq Cr^{-2} \big(e^{-c_1 \mu_1^{\alpha^2}}+e^{-\cbeta\mu_2^{\alpha}} \big),
\end{eqnarray}
where $c_1,\cbeta$ are independent of $\beta,\mu$.
Moreover,
\begin{eqnarray} \label{estimate-b1f}
\big\|A(\frac{\beta D_0}{\mu_2}) b(\frac{y-y_2}{R}) (1-b_1) f_i \big\|_0 \leq \|f\|_0 \leq e^{-\mu_1^{\alpha}} \leq  e^{-\mu_1^{\alpha^2}}.
\end{eqnarray}

We assume $\cbeta \leq 3^{-\alpha}$ without loss of generality, and choose $\beta=\cbeta^{-1/\alpha}$ so that the requirement $\beta\geq 3$ is satisfied.
The estimates \eqref{estimate-LDbu}, \eqref{estimate-Pib1}
and \eqref{estimate-b1f} show that the right-hand side of \eqref{formula-u2} satisfies the assumptions of Lemma \ref{lemma_exp1} for $\mu=\cbeta^{1/\alpha}\mu_2$ with the choice
\begin{equation}
\mu_2=\min\big\{ c_0^{-1}, (\frac{c_1}{\cbeta})^{\frac{1}{\alpha}} \big\} \, \mu_1^{\alpha}.
\end{equation}
The support condition \eqref{psi-support} is satisfied by $u_{\cdot,2}$ in $\Omega_0$ by choosing $\tilde{\psi}=\psi-\psi(y_2)$ due to Assumption \ref{assumption1}. 
This can be argued in the same way as for $u_{\cdot,1}$, considering the fact that $u_{\cdot,2}=0$ on $\Upsilon\cup B_{r/2}(y_1)$ by definition \eqref{def-uij}.
Hence applying Lemma \ref{lemma_exp1} to $u_{\cdot,2}$ gives
\begin{equation}\label{iteration_j=2}
\big\|A(\frac{D_0}{\omega}) b(\frac{y-y_2}{r}) u_{\cdot,2} \big\|_{H^1} \leq C_2\exp(-c_1\cbeta^{\alpha}\mu_2^{\alpha^2}), \quad \forall \omega\leq \cbeta \mu_2^{\alpha}/c_0.
\end{equation}

In the same way, the estimates for all $u_{\cdot,j}$ can be obtained by induction. We can choose $\mu=\mu_1>c_2$ sufficiently large such that $\mu_j>1$ for all $j=1,\cdots,N$, where $c_2$ depends on $N$. Recall that $N$ is bounded by \eqref{bound-N}.
For the $H^1$-norm of $u_{\cdot,j}$, we have
\begin{eqnarray*}
\|u_{\cdot,j}\|_{H^1} =\Big\|\Big(\prod_{k=1}^{j-1} (1-b_k) \Big) u \Big\|_{H^1} \leq C N r^{-1}.
\end{eqnarray*}
All relevant constants can be explicitly calculated as in Theorem 2.7 in \cite{BKL2}. In our case, the constants additionally depend on $m$, $\max_i \|v_i\|_{C^1}$ and $\|L\|_{\infty}$.
\end{proof}

Lemma \ref{iteration} yields the following global stability estimate.

\begin{proposition} \label{uc-stability}
Let $\Omega$ be a bounded connected open subset of $\mathbb{R}\times \mathbb{R}^n$. Suppose $u=(u_1,\cdots,u_m)$, $u_i \in H^1(\Omega)$ is a solution of the system of hyperbolic equations \eqref{eq_system} with $f=(f_1,\cdots,f_m),\,f_i \in L^2(\Omega)$.
Assume Assumption \ref{assumption1} is satisfied.
Then we have
\begin{eqnarray*}
 \|u\|_{L^{2}(\overline{\Omega}_a)} \le C \frac{\|u\|_{H^{1}(\Omega_0)}}{\Big(\log\big(1+\frac{\|u\|_{H^{1}(\Omega_0)}}{\|f\|_{L^{2}(\Omega_0)}}\big)\Big)^{\theta}}\,,
\end{eqnarray*}
where $\theta\in (0,1)$ is arbitrary. The constant $C$ explicitly depends on the $C^{2,\rho}$-norm of $\psi$, $\inf_{\Omega_0}|\psi'|$, $\min_i \inf_{\Omega_0}|p_i(\cdot,\psi')|$, $\dist(\partial \Omega_0, \Omega_a)$, $m$, $\max_i \|v_i\|_{C^1}$, $\|L\|_{\infty}$, $\theta$ and geometric parameters.
\end{proposition}

\begin{proof}
Without loss of generality, assume $\|u\|_{H^1(\Omega_0)}=1$.
If $\|f\|_0\geq e^{-c_2}$, then the inequality above satisfies trivially. Otherwise $\|f\|_0=e^{-\mu^{\alpha}}$ for some $\mu>c_2$, where $\alpha\in (0,1)$ is fixed. The estimates for the lower temporal frequencies follow from Lemma \ref{iteration}. Higher temporal frequencies can be estimated uniformly in frequency. Then the $\log$-stability estimate follows by the same argument as Theorem 1.1 in \cite{BKL2}.
\end{proof}

In the same way, we also have the following stability estimate for multiple domains, analogous to Theorem 1.2 in \cite{BKL2}.

\begin{proposition}\label{global}
Let $\Omega$ be a bounded connected open subset of $\mathbb{R}\times \mathbb{R}^n$. Suppose $u=(u_1,\cdots,u_m)$, $u_i \in H^1(\Omega)$ is a solution of the system of hyperbolic equations \eqref{eq_system} with $f=(f_1,\cdots,f_m),\,f_i \in L^2(\Omega)$. In $\Omega$, we assume the existence of a finite number of connected open subsets $\Omega_{j}^0$ and $\Omega_{j}$, $j=1,2,\dots,J$, a connected set $\Upsilon$ and functions $\psi_j$ satisfying the following assumptions.
\begin{enumerate}[(1)]
\item $\psi_j\in C^{2,\rho}(\Omega)$ for some $\rho\in (0,1]$; $\psi'_j(y)\neq 0$, $p_i(y,\psi'_{j}(y))\neq 0$ for all $i,j$ and all $y\in \Omega_{j}^0$, where $p_i$ denotes the principle symbol of the wave operator $P_i$ in \eqref{def-wave}.
\item ${\rm supp}(u)\cap \Upsilon=\emptyset$; there exists $\psi_{max,j}\in\mathbb{R}$ such that $\emptyset\neq\{y\in \Omega_{j}^0: \psi_j(y)> \psi_{max,j}\}\subset \overline{\Upsilon}_j$, where $\Upsilon_j=\Omega_{j}^0\cap(\cup_{l=1}^{j-1}\Omega_l\cup \Upsilon)$.
\item $\Omega_j=\{y\in \Omega_{j}^0-\overline{\Upsilon}_j: \psi_j(y) > \psi_{min,j}\}$ for some $\psi_{min,j}\in\mathbb{R}$, and $\dist(\partial\Omega_{j}^0,\Omega_j)>0$.
\item $\overline{\Omega}_a$ is connected, where $\Omega_a=\cup_{j=1}^J \Omega_j$.
\end{enumerate}
Then the following estimate holds for $\Omega_a$ and $\Omega^0=\cup_{j=1}^J\Omega_{j}^0$:
$$\|u\|_{L^2(\overline{\Omega}_a)}\leq C \frac{\|u\|_{H^1(\Omega^0)}}{\Big(\log\big(1+\frac{\|u\|_{H^1(\Omega^0)}}{\|f\|_{L^2(\Omega^0)}}\big)\Big)^{\theta}}\, ,$$
where $\theta\in(0,1)$ is arbitrary. The dependency of the constant $C$ is the same as Proposition \ref{uc-stability}.
\end{proposition}

\section{Stability of the unique continuation on Riemannian manifolds}
\label{sec-manifold}

From Proposition \ref{global}, one can obtain an explicit stability estimate for the unique continuation on Riemannian manifolds in a similar way as Theorem 3.3 in \cite{BKL2} or Theorem 3.1 in \cite{BILL}. The following is a proof of Theorem \ref{uc-manifold} which is analogous to Theorem 3.1 in \cite{BILL}.

\begin{proof}[Proof of Theorem \ref{uc-manifold}]
The proof is only a slight modification of the proof of Theorem 3.1 in \cite{BILL}.
We consider the submanifold $M-U$ and its boundary has two (smooth) connected components $\partial M,\,\partial U$. We take $\Gamma=\partial U$ in \cite{BILL} and follow the proof of Theorem 3.1 in \cite{BILL} by using Proposition \ref{global}.
Notice that only the first condition in Proposition \ref{global} is affected by the change of wave speed in the wave operator. Hence we only need to check that the domains $\Omega_j^0$ are non-characteristic with respect to $P_i$ for all $i$.

The $\psi_j$-functions constructed in \cite{BILL} in the simplest form are
$$\psi_j(x,t)=\big(T-d(x,z_j)\big)^2-t^2,$$
where $z_j\in M$ are fixed points. In our case of different wave speeds, one can choose the $\psi_j$-functions as follows:
\begin{equation}
\tilde{\psi}_j(x,t)=\big(T-v^{-1}d(x,z_j) \big)^2-t^2,\quad \textrm{for }x\in M,\,t\in [-T,T], 
\end{equation}
where
\begin{equation}\label{def-v}
v:=\min_i \inf_{x\in M} v_i(x)>0. \quad \textrm{(i.e. the minimal wave speed in $M$)}
\end{equation}
The domains $\Omega_j^0$ can be similarly defined as suitable level sets 
\begin{equation}
\Omega_j^0=\{(x,t)\in M \times [-T,T]:\tilde{\psi}_j(x,t)> h\}
\end{equation}
for some small positive parameter $h$.

Then it is straightforward to check that $\tilde{\psi}_j$ is non-characteristic in $\Omega_j^0$ with respect to all $P_i$. Namely, for any $i$,
\begin{eqnarray*}
-p_i \big((x,t),\nabla \tilde{\psi}_{j} \big) &=& v_i(x)^2 \sum_{k,l=1}^n g^{kl} (\partial_{x_k}\tilde{\psi}_j)( \partial_{x_l} \tilde{\psi}_j)-|\partial_t \tilde{\psi}_j|^2 =v_i(x)^2 |\nabla_x \tilde{\psi}_j|^2-|\partial_t \tilde{\psi}_j|^2 \nonumber\\
&=& 4\frac{v_i(x)^2}{v^2} \big(T-v^{-1}d(x,z_j)\big)^2-4t^2 \nonumber \\
&\geq& 4 \tilde{\psi}_j (x,t) > 4h,
\end{eqnarray*}
where we have used the fact that $|\nabla_x d(x,z_j)|=1$ and
$$\nabla_x \tilde{\psi}_j=-2(T-v^{-1}d(x,z_j)) v^{-1} \nabla_x d(x,z_j).$$
Recall that the matrix $(g^{jk})$ in the wave operator \eqref{def-wave} is the inverse of the matrix $(g_{jk})$ that is the Riemannian metric in local coordinates.

In general, we can choose the $\psi_j$-functions in the present case in a form similarly modified from (3.30) in \cite{BILL}:
\begin{equation}  \label{def-psiij}
\psi_{i,j}(x,t)=\bigg(\Big(1-\xi \big(d(x,\partial M\cup \partial U)\big)-\xi \big(\rho_0-d_h^s(x,z_{i,j})\big)\Big)T_i- v^{-1} d_h^s(x,z_{i,j})\bigg)^2-t^2,
\end{equation}
where $\xi\in C^{2,1}(\mathbb{R}_{\geq 0})$ is a decreasing function supported in $[0,h]$, and $d_h^s(\cdot,z_{i,j})$ is a smoothening of a distance function.
Differentiating $\psi_{i,j}$ with respect to $x$ gives
\begin{eqnarray*}
\nabla_x\psi_{i,j}&=&2\Big(\big(1-\xi(d(x,\partial M \cup \partial U))-\xi(\rho_0-d_h^s(x,z_{i,j}))\big)T_i- v^{-1} d_h^s(x,z_{i,j})\Big) \\
&&\big(-\xi^{\prime}T_i\nabla_x d(x,\partial M \cup \partial U) +\xi^{\prime}T_i\nabla_x d_h^s(x,z_{i,j})-v^{-1} \nabla_x d_h^s(x,z_{i,j})\big).
\end{eqnarray*}
By the construction in \cite{BILL}, $\nabla_x d_h^s(x,z_{i,j})$ and $\nabla_x d(x,\partial M)$ are of opposite directions when $x$ is near the boundary. Since $\xi'\leq 0$, the latter multiplier has length at least $v^{-1}$. Thus by the same calculation above for $\tilde{\psi_j}$, one can still show that the suitable level set of $\psi_{i,j}$ is non-characteristic with respect to all $P_i$.

The additional coefficient $v^{-1}$ can be understood as a time dilation by a factor of $v$. Thus Theorem 3.1 in \cite{BILL} gives the estimate in the following domain (see (2.6) in \cite{BILL}):
\begin{equation*}
\Omega_v (h)=\big\{(x,t)\in (M-U)\times [-T,T]: vT-|vt|-d_{M-U}(x,\partial U) >\sqrt{h},\; d_{M-U}(x,\partial M)>h \big\},
\end{equation*}
where $d_{M-U}$ is the Riemannian distance of the submanifold $M-U$.
Notice that $d_{M-U}(x,\partial U)=d(x,\partial U)$ for any $x\in M-U$, since any path from $x$ to the interior of $U$ must cross $\partial U$. In \eqref{Omega-UTv}, we actually used a smaller domain due to $d_{M-U}(\cdot,\partial M)\geq d(\cdot,\partial M)$. We note that the factor $v^{-1}$ is only added to one distance term in \eqref{def-psiij}. The other two distance terms there control how close this process approximates the optimal domain and therefore a constant factor is inconsequential to the final error estimate.

\smallskip
Finally we turn to the data on $U$. In our case, Lemma 3.3 and 3.4 in \cite{BILL} are not necessary as we have the whole manifold $M$ and functions $u_i\in H^1(M\times [-T,T])$ readily available. We can simply take the functions $u_i$ on $M$ and cut it off near $\partial U$ in $U$. More precisely, take a smooth function $\eta:M\to \R$ such that $\eta=1$ on $M-U$, $0\leq \eta \leq 1$ on $U$ within distance $h$ from $\partial U$, otherwise $0$. Consider $\tilde{u}_i=\eta u_i$. Then $\tilde{u}_i$ satisfies the following equation similar to \eqref{eq-local-u}:
\begin{eqnarray*}
P_i \tilde{u}_i+L_i(D \tilde{u}, \tilde{u})= \eta f_i +[P_i,\eta]u_i +\sum_{k=1}^m \sum_{l=0}^n L_{i;kl} (D_l \eta) u_k \,=:\tilde{f}_i.
\end{eqnarray*}
It is clear that $\|\tilde{u}_i\|_{H^1(M\times [-T,T])}\leq Ch^{-1} \|u_i\|_{H^1(M\times [-T,T])}$, and
\begin{eqnarray*}
\|\tilde{f}_i\|_{L^2(M\times [-T,T])}&\leq&\|f_i\|_{L^2((M-U) \times [-T,T])}+\|\tilde{f}_i\|_{L^2(U\times [-T,T])} \\
&\leq& \|f_i\|_{L^2(M\times [-T,T])}+ Ch^{-2}\|u\|_{H^1(U \times [-T,T])}.
\end{eqnarray*}
Applying these estimates to the last step of the proof of Theorem 3.1 in \cite{BILL} gives the desired estimate.
As for the dependency of the constant $C$, besides what is stated in the theorem, the constant also depends on geometric parameters: $\diam(M),\|R_M\|_{C^4},\|S_{\partial M}\|_{C^4}, \|S_{\partial U}\|_{C^4}$,  $\textrm{inj}(M-U), r_{\textrm{CAT}}(M-U), {\rm vol}_n(M), {\rm vol}_{n-1}(\partial M), {\rm vol}_{n-1}(\partial U)$.
\end{proof}

\begin{remark}
In Theorem \ref{uc-manifold}, we only use $H^1$-norm data on the interior domain $U$, instead of the higher regularity $H^{2,2}$-norm (see (2.7) in \cite{BILL}) on a subset of the manifold boundary used in \cite{BILL}. The extra regularity in \cite{BILL} were used to extend the data on the boundary to an extension of the manifold. In our present case, we are avoiding this issue by using data on an interior domain of the manifold. Nonetheless, the same method is valid for the system of hyperbolic equations with data on the boundary, which yields a similar $\log$-type stability estimate with $H^{2,2}$-norm boundary data.
\end{remark}

\begin{remark} \label{remark-domains}
In particular, if all wave speeds $v_i\equiv 1$ (or a constant), the domain $\Omega_v(h)$ defined in \eqref{Omega-UTv}, where we have quantitative estimates, can be arbitrarily close to the optimal domain given by Tataru's unique continuation theorem for the scalar wave equation in \cite{Tataru1}, since $h$ is a small parameter chosen in advance.
Recall that the optimal domain is known as the \emph{double cone of influence} defined as
$$K(U,T)=\big\{(x,t)\in M\times [-T,T] : d(x,U)< T-|t| \big\}.$$
However, if one wants the estimate in Theorem \ref{uc-manifold} to work closer to the optimal domain, the cost is the constant $\exp(h^{-cn})$ which goes up exponentially.

In the case of $1+1$ dimension and $v_i\equiv 1$, the optimal domain and the domain $\Omega_v(h)$ are illustrated in Figure \ref{fig_domains}. In general, if the wave speeds are not constant, the domain $\Omega_v(h)$, propagating according to the slowest speed, can be significantly smaller than the optimal domain.
\end{remark}

Theorem \ref{uc-manifold} yields the following estimate on the initial value.

\begin{corollary}\label{uc-initial}
Let $(M^n,g)$ be a compact, orientable, smooth Riemannian manifold of dimension $n \geq 2$ with smooth boundary $\partial M$, and $U$ be a connected open subset of $M$ with smooth boundary $\partial U$. Assume $\overline{U}\cap \partial M=\emptyset$.
Suppose $u=(u_1,\cdots,u_m)$, $u_i \in H^1(M\times [-T,T])$ is a solution of the system of hyperbolic equations \eqref{eq_system} with $f=(f_1,\cdots,f_m)=0$. Let $v=\min_i \inf_{x\in M} v_i(x)>0$ be the minimal wave speed in $M$.
If
$$\|u\|_{H^1(M\times [-T,T])}\leq \Lambda_0,\quad \|u\|_{H^1(U\times [-T,T])}\leq \varepsilon_0,$$
then for sufficiently small $h$, we have
$$\|u(\cdot,0)\|_{L^2(\Omega_v (2h,0))}\leq C^{\frac{1}{3}} h^{-\frac{1}{3}} \exp(h^{-c n}) \frac{\Lambda_0}{\Big(\log\big(1+h\frac{\Lambda_0}{\varepsilon_0}\big)\Big)^{\frac16}}\, .$$
The domain $\Omega_v(h,0)$ is defined by
\begin{equation}
\Omega_v(h,0)=\big\{x\in M\setminus U: d(x,\partial U) < vT-\sqrt{h},\; d(x,\partial M)>\sqrt{h} \big\}.
\end{equation}
The constants $C,c$ are independent of $h$, and their dependency is stated in Theorem \ref{uc-manifold}.
\end{corollary}

\begin{proof}
This directly follows from interpolation and the trace theorem. 
See the proof of Corollary 3.9 in \cite{BILL} for more details.
\end{proof}

Due to the Sobolev embedding theorem, Theorem \ref{uc-manifold} implies the following stable continuation result on the whole manifold.

\begin{proposition} \label{uc-whole}
Let $(M^n,g)$ be a compact, orientable, smooth Riemannian manifold of dimension $n \geq 2$ with smooth boundary $\partial M$, and $U$ be a connected open subset of $M$ with smooth boundary $\partial U$. Assume $\overline{U}\cap \partial M=\emptyset$.
Suppose $u=(u_1,\cdots,u_m)$, $u_i \in H^1(M\times [-T,T])$ is a solution of the system of hyperbolic equations \eqref{eq_system} with $f=(f_1,\cdots,f_m)=0$. 
Assume $T>2({\rm diam}(M)+1)/v$, where $v=\min_i \inf_{x\in M} v_i(x)>0$ is the minimal wave speed in $M$. 
If
$$\|u\|_{H^1(M\times [-T,T])}\leq \Lambda_0,\quad \|u\|_{H^1(U\times [-T,T])}\leq \varepsilon_0,$$
then there exist constants $h_0,C,c>0$ such that for any $0<h<h_0$, we have
$$\|u\|_{L^2 ((M\setminus U)\times [-\frac{T}{2},\frac{T}{2}] )}\leq C \exp(h^{-c n}) \frac{\Lambda_0}{\Big(\log\big(1+h\frac{\Lambda_0}{\varepsilon_0}\big)\Big)^{\frac12}} + C\Lambda_0 h^{\frac{1}{n+1}} .$$
Furthermore, for any $\theta\in (0,1)$, by interpolation,
$$\|u\|_{H^{1-\theta}((M\setminus U)\times [-\frac{T}{2},\frac{T}{2}])}\leq C^{\theta}\exp(h^{-c n}) \frac{\Lambda_0}{\Big(\log\big(1+h\frac{\Lambda_0}{\epsilon_0}\big)\Big)^{\frac{\theta}{2}}}+ C^{\theta}\Lambda_0 h^{\frac{\theta}{n+1}} .$$
\end{proposition}

\begin{proof}
For $T>2({\rm diam}(M)+1)/v$ and $h<1$, we see that 
$$(M\setminus U) \times [-\frac{T}{2},\frac{T}{2}] \subset \Omega_v(h) \cup \mathcal{N}_h\, ,$$
where
$$\mathcal{N}_h :=\{x\in M: d(x,\partial M)\leq h\} \times [-\frac{T}{2},\frac{T}{2}].$$
Theorem \ref{uc-manifold} gives the $L^2$-estimate on $\Omega_v(h)$, and thus we only need to estimate the $L^2$-norm on $\mathcal{N}_h$.
Apply the Sobolev embedding theorem (e.g. Theorem 4.12 in \cite{AF}) to the space $M^n\times [-T,T]$ which satisfies the uniform cone condition (Definition 4.8 in \cite{AF}),
$$\|u\|_{L^{\frac{2(n+1)}{n-1}}(M\times [-T,T])} \leq C\|u\|_{H^1(M\times [-T,T])}\leq C\Lambda_0.$$
Then,
$$\|u\|_{L^2(\mathcal{N}_h)} \leq \|u\|_{L^{\frac{2(n+1)}{n-1}}(M\times [-T,T])} \big(\textrm{vol}(\mathcal{N}_h) \big)^{\frac{1}{n+1}} \leq C \Lambda_0 h^{\frac{1}{n+1}}.$$
\end{proof}

\begin{proof}[Proof of Corollary \ref{uc-whole-loglog}]
Choose $h$ such that the two terms on the right-hand side of the $L^2$-estimate in Proposition \ref{uc-whole} are equal, and we get
\begin{equation}
h=C \big(\log |\log \varepsilon_0| \big)^{-c},
\end{equation}
for some constant $c$ depending only on $n$, and for some constant $C$ independent of $h$.
The condition $h<h_0$ gives the choice for $\widehat{\varepsilon_0}$:
\begin{equation}
\widehat{\varepsilon_0}=\Big( \exp \exp \big(C^{-1}h_0^{-1/c} \big) \Big)^{-1}.
\end{equation}
\end{proof}


\section{Application to fault dynamics}
\label{section-fault}

In this section, let $M^3\subset \mathbb{R}^3$ be a compact domain of dimension $3$ with smooth boundary representing the solid Earth.
Let $\Surf_\frc$ be a ($2$-dimensional) rupture surface. Assume that $\Surf_\frc$ is connected, orientable, smooth with Lipschitz boundary and $\overline{\Surf_\frc}\cap \partial M=\emptyset$. The open set $V$ is the observation domain satisfying $V\subset M\setminus \Surf_\frc$, see Figure \ref{fig_intro}. The set $U$ is a connected open subset of $V$ satisfying $\overline{U}\subset V$. Then it follows that $\overline{U} \cap (\overline{\Surf_\frc}\cup\partial M) =\emptyset$.

\begin{figure}[h]
  \begin{center}
    \includegraphics[width=0.4\linewidth]{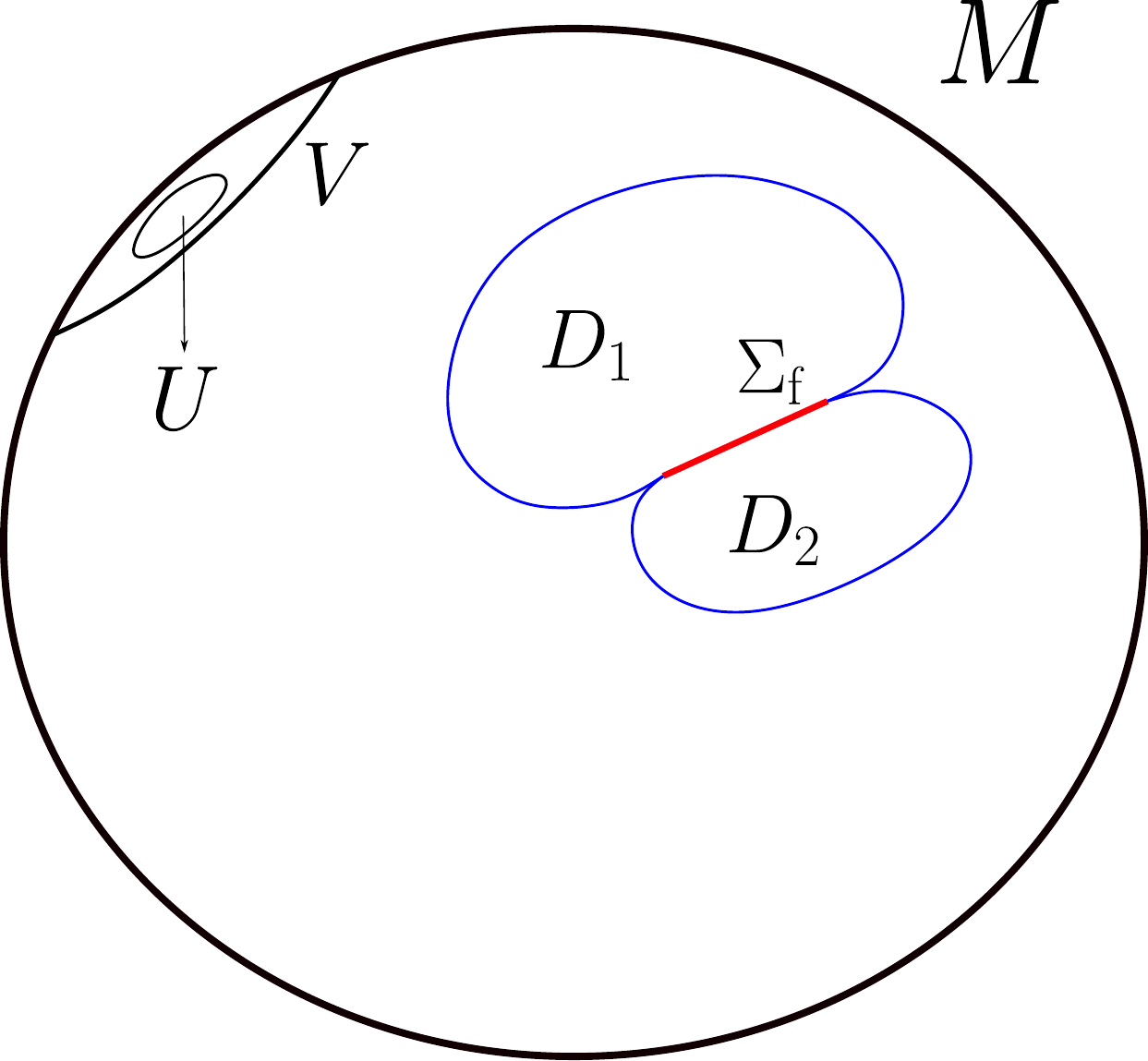}
    \caption{Rupture surface $\Surf_\frc$ under Assumption \ref{assumption2}. The set $U$ is a connected open subset of the observation domain $V$ satisfying $\overline{U}\subset V$.}
    \label{fig_assumption}
  \end{center}
\end{figure}

\subsection{Unique continuation to rupture surface}
We apply our stability results in Sections $\ref{section-uc}$ and $\ref{sec-manifold}$ to seismic waves. For a concise formulation we adopt the following setting.

\begin{assumption} \label{assumption2}
{\rm Suppose that there are two disjoint open subsets $D_1,D_2\subset M$ satisfying $\overline{D_1}\cup \overline{D_2}\subset {\rm int}(M)$ and $\Surf_\frc \subset \partial D_1 \cap \partial D_2$, such that the following condition holds.}

\smallskip
\noindent $(\ast)$ For $j=1,2$, $M_j=M\setminus \overline{D_j}$ is a connected open subset of $M$ with smooth boundary.
\end{assumption}

Under Assumption \ref{assumption2}, it follows that $M_1\cap M_2$ contains an open set of $M$, and
\begin{equation}\label{M1M2}
\Surf_\frc \subset \partial M_1\cap \partial M_2,\quad M=M_1\cup M_2.
\end{equation}
In other words, the rupture surface $\Surf_\frc$ can be approached from both sides, and it can be extended on either side into the boundary of a smooth submanifold. In practice one can try to construct the subsets $D_1,D_2$ to be open topological (3-dimensional) balls with smooth boundary, such that their closures $\overline{D_1},\overline{D_2}$ do not intersect $\partial M$. If such $D_1,D_2$ can be constructed, then the condition $(\ast)$ is satisfied.

\smallskip
We consider function spaces on the disjoint union $M_1\sqcup M_2$ instead of on $M$.
We say a function $u\in H^s(M_1\sqcup M_2)$ if $u|_{M_j}\in H^s(M_j)$ for $j=1,2$.
We define the $H^s$-norm on $M_1\sqcup M_2$ by
\begin{equation}\label{def-Hs}
\|u\|_{H^s(M_1\sqcup M_2)}^2:= \sum_{j=1}^{2} \big\|u|_{M_j} \big\|_{H^s(M_j)}^2,
\end{equation}
and for $T> 0$,
\begin{equation}\label{def-Hs}
\|u\|_{H^s ((M_1\sqcup M_2)\times [-T,T])}^2:= \sum_{j=1}^{2} \big\|u|_{M_j\times [-T,T]} \big\|_{H^s(M_j \times [-T,T])}^2.
\end{equation}

The seismic wave $\veu$ is modeled by the following equation of motion
\begin{equation} \label{eq-seismic}
\rho \partial_t^2 \veu -\nabla\cdot(\tsLaINI:\nabla\veu)=0 \quad \textrm{in } M \setminus \overline{\Surf_\frc}\, ,
\end{equation}
where the prestressed elasticity tensor $\tsLaINI$ is related to the \textit{in situ}
 isentropic stiffness tensor $\tsC$ by
\[
    \itLaINI_{ijkl}=C_{ijkl}+\frac12\Big(
     (T^0)_{ij}\delta_{kl}
    +(T^0)_{kl}\delta_{ij}
    +(T^0)_{ik}\delta_{jl}
    -(T^0)_{il}\delta_{jk}
    -(T^0)_{jk}\delta_{il}
    -(T^0)_{jl}\delta_{ik}
    \Big) ,
\]
and the operation $:$ is defined as $(\tsLaINI:\nabla\veu)_{ij}=\sum_{k,l} \itLaINI_{ijkl} \partial_l u_k$ in components.
In the case of isotropy and hydrostatic prestress $\tsT^0 = -p^0
\boldsymbol{I}$, the prestressed elasticity tensor $\tsLaINI$ has the form
\begin{equation} \label{T0-simple}
   \itLaINI_{ijkl} = \lambda \delta_{ij} \delta_{kl}
   + \mu (\delta_{ik} \delta_{jl} + \delta_{il} \delta_{jk})
   - p^0 (\delta_{ij} \delta_{kl} - \delta_{il} \delta_{jk}) .
\end{equation}
In this case the equation \eqref{eq-seismic} has the same form as the classical elasticity system in \eqref{eq-elastic}, namely
\begin{equation}\label{eq-seismic-wave}
\rho \partial_t^2 \veu -\mu\Delta \veu-(\lambda+\mu)\,\nabla{\rm div} \, \veu+\,\textrm{first order terms}=0.
\end{equation}
Note that $p^0$ appears in the first order terms of \eqref{eq-seismic-wave}, see e.g. \cite[Section 2.2]{SLXSD}.
We assume that
\begin{equation}
\rho,\, \mu,\, \lambda,\, p^0 \in C^{\infty}(M\setminus \overline{\Surf_\frc}) \,\textrm{ are time-independent}.
\end{equation}

Due to Lemma \ref{elastic-to-system}, the equation \eqref{eq-seismic} with \eqref{T0-simple} can be written in the form of the system of hyperbolic equations \eqref{eq_system}, and therefore our results in Sections $\ref{section-uc}$ and $\ref{sec-manifold}$ apply.
We consider the unique continuation in each smooth manifold $M_j$ $(j=1,2)$ with smooth boundary assumed in Assumption \ref{assumption2}. 
We observe on a connected open subset $U\subset M_1\cap M_2$ satisfying $\overline{U} \cap \partial M_j=\emptyset$. Thus we can apply Corollary \ref{uc-whole-loglog} to each manifold $M_j$ with the open set $U$.

\begin{theorem}\label{uc-fault}
Let $M^3$ (the solid Earth), $\Surf_\frc$ (the rupture surface) be defined at the beginning of Section \ref{section-fault}.
Let $M_j$ $(j=1,2)$ be the submanifolds with smooth boundary as in Assumption \ref{assumption2},
and $\veu=(u_1,u_2,u_3),\, u_i\in H^2((M_1\sqcup M_2)\times [-T,T])$ be a seismic wave satisfying \eqref{eq-seismic} with \eqref{T0-simple}.
We observe on a connected open subset $U\subset M_1\cap M_2$ with smooth boundary satisfying $\overline{U} \cap \partial M_j=\emptyset$. 
Assume $T>2\big(\max_j {\rm diam}(M_j)+1\big)/v$, where $v=\inf_M \sqrt{\mu/\rho}$ is the minimal wave speed.
If
$$\|\veu\|_{H^2((M_1\sqcup M_2)\times [-T,T])}\leq \Lambda_0,\quad \|\veu\|_{H^2(U \times [-T,T])}\leq \varepsilon_0,$$
then there exist constants $\widehat{\varepsilon_0},C,c$ such that for any $0<\varepsilon_0<\widehat{\varepsilon_0}$, we have
$$\big\|(\veu, {\rm div}\, \veu, {\rm curl}\, \veu) \big\|_{L^2 ((M_j \setminus U)\times [-\frac{T}{2},\frac{T}{2}] )}\leq C \big(\log |\log \varepsilon_0| \big)^{-c} ,$$
where $C$ is independent of $\varepsilon_0$, and $c$ is an absolute constant.
Furthermore, for any $\theta\in (0,1)$, by interpolation,
$$\big\|(\veu, {\rm div}\, \veu, {\rm curl}\, \veu) \big\|_{H^{1-\theta}((M_j \setminus U)\times [-\frac{T}{2},\frac{T}{2}])}\leq C \big(\log |\log \varepsilon_0| \big)^{-\theta c}.$$
\end{theorem}

\begin{proof}
By Lemma \ref{elastic-to-system}, the vector-valued function 
\begin{equation*} \label{vector-F}
{\bf F}:=(\veu, {\rm div}\, \veu, {\rm curl}\, \veu)
\end{equation*}
satisfies the system of hyperbolic equations \eqref{eq_system}.
The bounds on the $H^2$-norm of $\veu$ give the bounds
\begin{equation*} \label{norm-F}
\|{\bf F}\|_{H^1(M_j\times [-T,T])}\leq \Lambda_0,\quad \|{\bf F}\|_{H^1(U \times [-T,T])}\leq \varepsilon_0.
\end{equation*}
Then applying Corollary \ref{uc-whole-loglog} yields the result.
\end{proof}

\begin{remark}
The constant $C$ in Theorem \ref{uc-fault} depends on the geometric parameters of $M_j$ assumed in Assumption \ref{assumption2}.
\end{remark}

\noindent \textbf{The trace onto rupture surface.}
Recall that Assumption \ref{assumption2} indicates $\Surf_\frc\subset \partial M_1\cap \partial M_2$.
In the boundary normal neighborhoods of $M_1,M_2$, the two sides of $\Surf_\frc$ are product spaces $\Surf_\frc\times [0,{\rm inj}(M_1)/2]$ and $\Surf_\frc\times [0,{\rm inj}(M_2)/2]$, where ${\rm inj}(M_j)$ is the injectivity radius of $M_j$. Theorem \ref{uc-fault} gives an estimate on the $H^{1-\theta}$-norm of $\veu$ on two sides of $\Surf_\frc$, more precisely, on $\Surf_\frc\times [0,{\rm inj}(M_1)/2]\times [-T/2,T/2]$ and $\Surf_\frc\times [0,{\rm inj}(M_2)/2]\times [-T/2,T/2]$.
Hence the trace theorem yields that the trace of $\veu$ onto $\Surf_\frc$ is well-defined from both sides in $H^{\kappa}(\Surf_\frc \times [-T/2,T/2])$ for $\kappa\in (0,1/2)$. Namely, writing the trace of $\veu$ onto $\Surf_\frc$ from the two sides as
\begin{equation} \label{def-traceu}
\veu_{\pm}:= \lim_{h\to 0^{\pm}} \veu (z+h\ven,t),\quad z\in \Surf_\frc,
\end{equation}
the trace theorem and Theorem \ref{uc-fault} yield that, for any $\kappa\in (0,1/2)$,
\begin{eqnarray} \label{estimate-trace}
\|\veu_{\pm}\|_{H^{\kappa}(\Surf_\frc \times [-\frac{T}{2},\frac{T}{2}])} &\leq& \max_j \, C(\kappa,M_j) \|\veu\|_{H^{\kappa+\frac12}((M_j\setminus U)\times [-\frac{T}{2},\frac{T}{2}])} \nonumber \\
&\leq& C \big(\log |\log \varepsilon_0| \big)^{-(\frac12-\kappa)c},
\end{eqnarray}
where $C$ is independent of $\varepsilon_0$.

\smallskip
\subsection{Kinematic inverse rupture problem}

Now we show that we can determine the displacement, $\veu$, and traction, $\vetau_1$, on
both sides of the rupture surface $\Surf_\frc$ by the unique
continuation. By implication, we obtain the tangential jump of particle displacement, $[\veu_\TT]_-^+$ and friction force, $\vetau_\frc$.

On the (orientable) rupture surface $\Surf_\frc$ with unit normal vector
$\ven$, the dynamic slip boundary condition and the force equilibrium
are satisfied, which gives
\begin{equation}
    \left\{
        \arraycolsep=1.4pt\def\arraystretch{1.7}
        \begin{array}{rl}
            \jmp{\ven\cdot\veu} = & 0 ,
            \\
            \jmp{\vetau_1(\veu)+\vetau_2(\veu)} = & 0 ,
            \\
            \vetau_\frc- \big(\ven\cdot\tsT^0
            +\vetau_1(\veu)+\vetau_2(\veu) \big)_\TT = & 0 ,
        \end{array}
        \right.
        \quad \mbox{on } \fSurf,
    \label{eq:fric bc}
\end{equation}
with
\begin{equation}
    \left\{
        \arraycolsep=1.4pt\def\arraystretch{1.7}
        \begin{array}{rl}
            \vetau_1(\veu)=&\ven\cdot(\tsLaINI:\nabla\veu) ,
            \\
            \vetau_2(\veu)=&
            -\nablaS\cdot \big(\veu(\ven\cdot\tsT^0) \big) ,
        \end{array}
        \right.
    \label{eq:fric bc var}
\end{equation}
both of which are linearly depending on $\veu$, and the surface
divergence is defined by $\nablaS\cdot\vef
=\nabla\cdot\vef-(\nabla\vef\cdot\ven)\cdot\ven$. 
Here the labels ``$+$'' and ``$-$'' indicate the two sides of $\Surf_\frc$, and the subscript ${}_\TT$ represents the tangential component with respect to $\Surf_\frc$.
On the exterior boundary $\partial M$, with unit normal vector $\boldsymbol{\nu}$ of the domain $M$,
we apply the boundary condition
\[
   \boldsymbol{\nu}\cdot(\tsLaINI:\nabla\veu)
   = \boldsymbol{\nu}\cdot\tsT^0 = 0 .
\]
In the above, $\tsT^0$ is known and we assume that  the components of $\tsT^0$ are time-independent smooth functions on $M\setminus \overline{\Surf_\frc}$. 
Denote by $\veu_{\pm}, (\vetau_1(\veu))_{\pm}, (\vetau_2(\veu))_{\pm}$ the traces of $\veu, \vetau_1(\veu), \vetau_2(\veu)$ on the two sides of the rupture surface $\Surf_\frc$, respectively.

\smallskip
With Theorem \ref{uc-fault}, we can solve the kinematic inverse rupture problem as follows.

\begin{theorem} \label{determine-fault}
Let $M^3$ (the solid Earth), $\Surf_\frc$ (the rupture surface) be defined at the beginning of Section \ref{section-fault}.
Let $M_j$ $(j=1,2)$ be the submanifolds with smooth boundary as in Assumption \ref{assumption2},
and $\veu=(u_1,u_2,u_3),\, u_i\in H^2((M_1\sqcup M_2)\times [-T,T])$ be a seismic wave satisfying \eqref{eq-seismic} with \eqref{T0-simple}.
We observe on a connected open subset $U\subset M_1\cap M_2$ with smooth boundary satisfying $\overline{U} \cap \partial M_j=\emptyset$. 
Then for sufficiently large $T$, we can determine 
$$\veu_{\pm}\in H^{\kappa}(\Surf_\frc \times [-\frac{T}{2},\frac{T}{2}]), \;\quad (\vetau_1(\veu))_{\pm}, \, (\vetau_2(\veu))_{\pm}, \, \vetau_\frc \in H^{\kappa-1}(\Surf_\frc \times [-\frac{T}{2},\frac{T}{2}]),$$
for any $\kappa\in (0,\frac12)$. 

Furthermore, if 
$$\|\veu\|_{H^2((M_1\sqcup M_2)\times [-T,T])}\leq \Lambda_0,\quad \|\veu\|_{H^2(U \times [-T,T])}\leq \varepsilon_0,$$
then there exist constants $\widehat{\varepsilon_0},C,c$ such that for any $0<\varepsilon_0<\widehat{\varepsilon_0}$, we have
\begin{equation*}
\|\veu_{\pm}\|_{H^{\kappa}(\Surf_\frc \times [-\frac{T}{2},\frac{T}{2}])} + \big\|\vetau_\frc-(\ven\cdot\tsT^0)_\TT \big\|_{H^{\kappa-1}(\Surf_\frc \times [-\frac{T}{2},\frac{T}{2}])} \leq C \big(\log |\log \varepsilon_0| \big)^{-c},
\end{equation*}
where $C$ is independent of $\varepsilon_0$, and $c$ depends only on $\kappa$.
\end{theorem}

\begin{proof}
With the unique continuation, we determine the displacement $\veu$ on both sides of the rupture surface $\veu_{\pm}$ defined by \eqref{def-traceu},
and, hence, $[\veu_\TT]_-^+$, the tangential jump of particle
displacement across the rupture surface, and
\[
   (\vetau_1(\veu))_{\pm} := \lim_{h\to 0^{\pm}} \vetau_1(\veu) (z+h\ven,t), \quad z\in \Surf_\frc.
\]
By implication, as $\tsT^0$ is known, we determine
\[
   (\vetau_2(\veu))_{\pm}:=\lim_{h\to 0^{\pm}} \vetau_2(\veu) (z+h\ven,t), \quad z\in \Surf_\frc.
\]
Thus we obtain $(\vetau_1(\veu))_{\TT,\pm}$ and
$(\vetau_2(\veu))_{\TT,\pm}$ and, hence, $\vetau_\frc$.

\smallskip
The regularity of $\veu_{\pm}$ was already discussed following Theorem \ref{uc-fault}, and the regularity estimate was given in \eqref{estimate-trace}.

For the regularity of $\vetau_\frc$, we recall the $H_{(k,s)}$-norm in $\R^{n+1}$ (see Definition B.1.10 in \cite{HIII}) defined as
\begin{equation} 
\|u\|^2_{(k,s)} = \int_{\R^{n+1}} |\widehat u(\xi)|^2 (1+|\xi|^2)^k (1+|\xi'|^2)^s d\xi \, ,
\end{equation}
with respect to the coordinates $y=(x',x^n)\in \R^{n}\times \R$. 
Note that when $s=0$, the $H_{(k,0)}$-norm above is equivalent to the usual $H^k$-norm.
In our case, the local coordinates can be chosen as the boundary normal coordinate of $M_j$ such that
\begin{align*}
\p M_j = \{x^n = 0\}, \quad P_i = \p_{x^n}^2 + a_i(y, D'),
\end{align*}
where $P_i$ is the wave operator \eqref{def-wave}. Moreover, define $\bar H_{(k,s)} = \{u|_{x^n > 0} : u \in H_{(k,s)}\}$.
Recall also Theorem B.2.9 in \cite{HIII} that allows us to trade smoothness from the tangential variables to the normal variable: if $u \in \bar H_{(k_1,s_1)}$ and $P_i u \in \bar H_{(k_2-2,s_2)}$, then 
$u \in \bar H_{(k,s)}$ if $k \le k_2$ and $k + s \le k_j + s_j$, $j=1,2$.

Now consider the homogeneous system \eqref{eq_system} with $f=0$ and $u \in H^\theta$ for some $\theta \in \R$. 
Suppose that the coefficients of $L_i$ are smooth.
We aim to show that $\p_{x^n} u|_{x^n = 0}$ is well-defined in a rough Sobolev space. 
Locally $u \in \bar H_{(\theta, 0)} = H^\theta$ and $P_i u_i = -L_i(Du,u)  \in \bar H_{(\theta-1,0)}$. 
This is due to $D u\in H^{\theta-1}$ and
\begin{eqnarray} \label{norm-Du-u}
\|D u\|^2_{(\theta-1,s)} &=& \int_{\R^{n+1}} |\xi \widehat u(\xi)|^2 (1+|\xi|^2)^{\theta-1} (1+|\xi'|^2)^s d\xi \nonumber \\
&\leq& \frac12 \int_{\R^{n+1}} |\widehat u(\xi)|^2 (1+|\xi|^2)^{\theta} (1+|\xi'|^2)^s d\xi\, =\frac12 \|u\|^2_{(\theta,s)}.
\end{eqnarray}
Thus $u_i \in \bar H_{(k,s)}$ if $k \le \theta+1$ and $k + s \le \theta$.
In particular, $u_i \in \bar H_{(\theta+1,-1)}$. For an estimate on the norm, define $X=\{v\in \bar H_{(\theta,0)}: \|v\|_X <\infty\}$ where $\|v\|_X := \|v\|_{(\theta,0)} + \|P_i v\|_{(\theta-1,0)}$. It follows from the closed graph theorem that $X$ is a Banach space. Then apply Lemma \ref{lemma-A2} with $Y=\bar H_{(\theta+1,-1)}$ and $Z=\bar H_{(\theta,0)}$, and using \eqref{norm-Du-u}, we have
\begin{eqnarray} \label{estimate-theta1-1}
\|u_i\|_{(\theta+1,-1)} &\leq& C \Big( \|u_i\|_{(\theta,0)}  + \|P_i u_i\|_{(\theta-1,0)} \Big) \nonumber \\
&\leq& C \Big( \|u_i\|_{(\theta,0)}  +  \| D u_i\|_{(\theta-1,0)} +\| u_i\|_{(\theta-1,0)} \Big) \leq C  \|u_i\|_{H^{\theta}}.
\end{eqnarray}
In fact the constant in \eqref{estimate-theta1-1} depends only on the coefficients of $P_i, L_i$, which can be extracted from the proof of Theorem B.2.9 in \cite{HIII}.
Repeating the argument above for all $i=1,...,m$ (which requires changing coordinates and using the coordinate invariant versions of the spaces),
we have $u \in \bar H_{(\theta+1,-1)}$. Then it follows from \cite[Theorem B.2.7]{HIII} that $\p_{x^n} u|_{x^n = 0}$ is well-defined in $H^{\theta-\frac32}$ as rough distributions if $\theta>\frac12$, and combining with \eqref{estimate-theta1-1}, 
\begin{equation}\label{estimate-norm-normal}
\big\|\p_{x^n} u|_{x^n = 0} \big\|_{H^{\theta-\frac32}} \leq C \|u\|_{(\theta+1,-1)} \leq C\|u\|_{H^{\theta}}.
\end{equation}

Since $\veu\in H^{\theta}(M_j\times [-T/2,T/2])$ for any $\theta\in (\frac12,1)$ by Theorem \ref{uc-fault}, the argument above and \eqref{eq:fric bc var} show that $(\vetau_1(\veu))_{\pm} \in H^{\theta-3/2}(\Surf_\frc \times [-T/2,T/2])$. 
The estimate on the norm of $(\vetau_1(\veu))_{\pm}$ is given by \eqref{estimate-norm-normal} and Theorem \ref{uc-fault}:
\begin{equation} \label{estimate-norm-tau1}
\big\|(\vetau_1(\veu))_{\pm} \big\|_{H^{\theta-\frac32}(\Surf_\frc \times [-\frac{T}{2},\frac{T}{2}])} \leq C \big(\log |\log \varepsilon_0| \big)^{-c}.
\end{equation}
For the regularity of $(\vetau_2(\veu))_{\pm}$, since $\veu_{\pm}\in H^{\kappa}(\Surf_\frc \times [-T/2,T/2])$ for $\kappa\in (0,\frac12)$, then $(\veu(\ven\cdot\tsT^0))_{\pm}\in H^{\kappa}$ and hence $(\vetau_2(\veu))_{\pm}\in H^{\kappa-1}(\Surf_\frc \times [-T/2,T/2])$ by \eqref{eq:fric bc var}.
The estimate on the norm of $(\vetau_2(\veu))_{\pm}$ is given by \eqref{estimate-trace} and \eqref{norm-Du-u}:
\begin{equation} \label{estimate-norm-tau2}
\big\|(\vetau_2(\veu))_{\pm} \big\|_{H^{\kappa-1}(\Surf_\frc \times [-\frac{T}{2},\frac{T}{2}])} \leq C \big(\log |\log \varepsilon_0| \big)^{-c}.
\end{equation}
Thus $\vetau_\frc \in H^{\kappa-1}(\Surf_\frc \times [-T/2,T/2])$ by \eqref{eq:fric bc}, and the estimate on its norm directly follows from \eqref{estimate-norm-tau1} and \eqref{estimate-norm-tau2}.
\end{proof}

Theorem \ref{determine-fault} yields the following corollary by interpolation.

\begin{corollary} \label{coro-higher-order}
Under the assumptions of Theorem \ref{determine-fault}, assume furthermore that
$$\|\veu\|_{H^s((M_1\sqcup M_2)\times [-T,T])}\leq \Lambda,\quad \textrm{for some }s\geq 2.$$
Then there exist constants $\widehat{\varepsilon_0},C,c$ such that for any $0<\varepsilon_0<\widehat{\varepsilon_0}$, we have
\begin{equation*}
\|\veu_{\pm}\|_{H^{r}(\Surf_\frc \times [-\frac{T}{2},\frac{T}{2}])} + \big\|\vetau_\frc-(\ven\cdot\tsT^0)_\TT \big\|_{H^{r-1}(\Surf_\frc \times [-\frac{T}{2},\frac{T}{2}])} \leq C \big(\log |\log \varepsilon_0| \big)^{-c},
\end{equation*}
for any $r\in (1,s-\frac12)$, where $C$ is independent of $\varepsilon_0$, and $c$ depends only on $r,s$.
\end{corollary}

\begin{proof}
Since we assume $\veu|_{M_j\times [-T,T]}\in H^s (M_j\times [-T,T])$ for $j=1,2$, the trace onto $\Surf_\frc$ from both sides $\veu_{\pm}\in H^{s-\frac12}(\Surf_\frc \times [-T,T])$, and
\begin{equation} \label{bound-higher-u}
\|\veu_{\pm}\|_{H^{s-\frac12}(\Surf_\frc \times [-T,T])} \leq C\|\veu\|_{H^{s}(M_j\times [-T,T])} \leq C \Lambda.
\end{equation}
Moreover, it follows from \eqref{eq:fric bc var} that $\vetau_1(\veu),\vetau_2(\veu) \in H^{s-1}$ in the boundary normal neighborhood of $M_j$. Hence their traces $(\vetau_1(\veu))_{\pm}, (\vetau_2(\veu))_{\pm}\in H^{s-\frac32}(\Surf_\frc \times [-T,T])$, and the norms are also bounded by $C\Lambda$. Therefore $\vetau_\frc\in H^{s-\frac32}(\Surf_\frc \times [-T,T])$ by \eqref{eq:fric bc}, and
\begin{equation} \label{bound-higher-tau}
\big\|\vetau_\frc-(\ven\cdot\tsT^0)_\TT \big\|_{H^{s-\frac32}(\Surf_\frc \times [-T,T])} \leq C \Lambda.
\end{equation}
Then the corollary follows by interpolating (e.g. \cite[Theorem 6.4.5]{BL}) between \eqref{bound-higher-u}, \eqref{bound-higher-tau} and Theorem \ref{determine-fault}.
\end{proof}

\begin{remark}\label{rem_balg}
In particular, when $s>2$, we obtain estimates from Corollary \ref{coro-higher-order} for $\veu_{\pm}$ on Sobolev spaces that are Banach algebras, see \cite[Theorem 4.39]{AF}.
\end{remark}

\medskip
\appendix

\section{}

Consider the classical elasticity system in a bounded domain $\Omega \subset \R\times \R^3$,
\begin{equation}\label{eq-elastic}
\rho \partial_t^2 u -\mu(\Delta u+\nabla {\rm div} \, u)-\nabla(\lambda\, {\rm div}\, u)-\sum_{j=1}^3 \nabla\mu \cdot (\nabla u_j +\partial_j u)e_j=0,
\end{equation}
for the displacement vector $u=(u_1,u_2,u_3)$ depending on $(t,x)\in \Omega$. Assume the density $\rho\in C^1(\overline{\Omega})$ and the Lam\'e parameters $\mu,\lambda\in C^2(\overline{\Omega})$.

\begin{lemma}[Lemma 5.1 in \cite{EINT}] \label{elastic-to-system}
Let $v={\rm div}\, u$, $w={\rm curl}\, u$. Assume $\rho\in C^1(\overline{\Omega})$ and $\mu,\lambda\in C^2(\overline{\Omega})$. If $u$ solves \eqref{eq-elastic} then
$$\frac{\rho}{\mu} \partial_t^2 u -\Delta u+A_{1}(u,v)=0,$$
$$\frac{\rho}{2\mu+\lambda} \partial_t^2 v -\Delta v+A_{2}(u,v,w)=0,$$
$$\frac{\rho}{\mu} \partial_t^2 w -\Delta w+A_{3}(u,v,w)=0,$$
where $A_1,A_2,A_3$ are linear differential operators of first order with coefficients in $C^0(\overline{\Omega})$. Moreover, when $\rho,\mu,\lambda$ do not depend on $t$, then the coefficients of $A_j$ do not depend on $t$.
\end{lemma}

Note that this linear system of equations above satisfied by the vector $(u, {\rm div}\, u, {\rm curl}\, u)$ consists of seven scalar equations. A similar result holds for the Maxwell system, see Lemma 4.1 in \cite{EINT}.

\smallskip
\begin{lemma} \label{lemma-A2}
Let $X \subset Y \subset Z$ be three Banach spaces. Suppose that there is a constant $C_0 > 0$ such that for all $x \in X$ and $y \in Y$,
$$
\|x\|_Z \le C_0 \|x\|_X,
\quad
\|y\|_Z \le C_0 \|y\|_Y.
$$
Then there is a constant $C > 0$ such that for all $x \in X$,
$$
\|x\|_Y \le C \|x\|_X.
$$
\end{lemma} 
\begin{proof}
Due to the closed graph theorem, applied to the inclusion map $I : X \to Y$, it is enough to show that if $x_n \to x$ in $X$ and $x_n \to y$ in $Y$ then $x = y$. The continuous inclusions $X \subset Z$
and $Y \subset Z$ imply that $x_n \to x$ and $x_n \to y$ in $Z$. Thus $x=y$.
\end{proof}

\end{document}